\documentclass[11pt,psamsfonts]{amsart}
\usepackage{amssymb, amsthm, mathrsfs}
\input xy
\xyoption{all}
\usepackage[left=1.0in,right=1.0in,top=1.0in,bottom=1.0in]{geometry}

\newtheorem{theorem}{Theorem}[section]
\newtheorem{lemma}[theorem]{Lemma}
\newtheorem{corollary}[theorem]{Corollary}
\newtheorem{proposition}[theorem]{Proposition}

\newtheorem*{thma}{Theorem A}
\newtheorem*{thmb}{Corollary B}
\newtheorem*{thmc}{Theorem C}
\newtheorem*{thmd}{Theorem D}

\theoremstyle{definition}
\newtheorem{notation}[theorem]{Notation}

\newtheorem{fact}[theorem]{Fact}
\newtheorem{example}[theorem]{Example}
\newtheorem{definition}[theorem]{Definition}
\newtheorem{remark}[theorem]{Remark}
\newtheorem{question}[theorem]{Question}

\newcommand\QED{\qed\medskip}
\newcommand\Tp{{\mathbb T}\hbox{\hskip2pt}}

\newcommand\T{{\mathbb T}}
\newcommand\Z{{\mathbb Z}}
\newcommand\N{{\mathbb N}}
\newcommand\Q{{\mathbb Q}}

\newcommand\Prm{{\mathbb P}}
\newcommand\B{\mathscr{B}}

\newcommand\Top{\mathcal{T}}

\newcommand\TT{{\mathfrak T}}
\newcommand\n{\mathfrak{n}}

\newcommand\Mar{\mathfrak M}
\newcommand\Zar{\mathfrak Z}
\newcommand\EE{{\mathfrak E}}

\newcommand\E{\mathscr{E}}
\newcommand\cont{\mathfrak{c}}

\newcommand\Round[1]{Almost ${#1}$-torsion}
\newcommand\round[1]{almost ${#1}$-torsion}
\newcommand\Zc[1]{${#1}$-atom}
\newcommand\CL{\mathrm{cl}}
\newcommand\fin[1]{{#1}^\cup}

\begin{document}

\title[]{The Markov-Zariski topology of an abelian group}

\author[]{Dikran Dikranjan}
\address[D.Dikranjan]{
Dipartimento di Matematica e Informatica\\
Universit\`{a} di Udine\\
Via delle Scienze  206, 33100 Udine\\
Italy}
\email{dikranja@dimi.uniud.it} 

\author[]
{Dmitri Shakhmatov}
\address[D. Shakhmatov]{\hfill\break
Division of Mathematics, Physics and Earth Sciences\\
Graduate School of Science and Engineering\\
Ehime University\\
Matsuyama 790-8577\\
Japan}
\email{dmitri@dpc.ehime-u.ac.jp}

\thanks{The first named author was partially supported by SRA grants P1-0292-0101 and J1-9643-0101, as well as grant MTM2009-14409-C02-01.}

\thanks{The second named author was partially supported by the Grant-in-Aid for Scientific Research (C) No.~19540092 of the Japan Society for the Promotion of Science (JSPS)}

\dedicatory{Dedicated to Kenneth Kunen on the occasion of his 65th anniversary}


\begin{abstract}
According to Markov \cite{Mar}, a subset of an abelian group $G$ of the form $\{x\in G: nx=a\}$,
for some integer $n$ and some element $a\in G$, is an elementary algebraic 
set; finite unions of elementary algebraic sets are called algebraic sets. We prove that a subset of an abelian group $G$ is  algebraic if and only if  it is closed in every precompact (=totally bounded) Hausdorff group topology on $G$. The family of all algebraic subsets of an abelian group $G$ forms the family of closed subsets of a unique Noetherian $T_1$ topology $\Zar_G$ on $G$ called the {\em Zariski\/}, or {\em verbal\/}, topology of  $G$ \cite{Bryant}.  We  investigate the properties of this topology. In particular, we show that the Zariski topology is always hereditarily separable and 
Fr\'echet-Urysohn.

For a countable family  $\mathscr{F}$ of subsets of an abelian group $G$ of cardinality at most the continuum, we construct a precompact metric group topology  $\mathcal{T}$ on $G$ such that  the  $\mathcal{T}$-closure of each member of $\mathscr{F}$ coincides with its $\Zar_G$-closure.  As an application, we  provide a characterization of the subsets of $G$ that are $\mathcal{T}$-dense in {\em some\/} Hausdorff group topology $\mathcal{T}$ on $G$, and we show that such a topology, if it exists, can always be chosen so that it is precompact and metric. This provides a partial answer to a long-standing problem of Markov \cite{Mar}.  
\end{abstract}

\maketitle

We use $\Prm$ and $\mathbb{N}$ to denote the sets of all prime numbers and  all natural numbers, respectively. In this paper, $0\in\mathbb{N}$. As usual, $\Z$ denotes the  group of integers, and $\Z(n)$ denotes the cyclic group of order  $n$.  We use $\cont$ to denote the cardinality of the continuum.
The symbol $\omega_1$ denotes the first uncountable cardinal.
\section{Introduction}

\subsection{Three topologies on a group}

In 1944, Markov \cite{Mar} introduced four special families of subsets of a group $G$:  

\begin{definition}
\label{Markov's:definition}
{\rm (\cite{Mar})} A subset $X$ of a group $G$ is called:
\begin{itemize}
  \item[(a)] {\em elementary algebraic} if there exist an integer $n>0$, elements $a_1,\ldots, a_n\in G$ and $\varepsilon_1,\ldots,\varepsilon_n\in\{-1,1\}$, such that $X=\{x\in G: x^{\varepsilon_1}a_1x^{\varepsilon_2}a_2\ldots  a_{n-1}x^{\varepsilon_n}a_n=1\}$,
  \item[(b)] {\em algebraic} if $X$ is an intersection of finite unions of  elementary algebraic subsets of $G$,
  \item[(c)] {\em unconditionally closed} if $X$ is closed in {\em every\/} Hausdorff group topology on $G$,
  \item[(d)] {\em potentially dense\/} if $G$ admits {\em some\/} Hausdorff  group topology $\mathcal{T}$ such that $X$ is dense in $(G,\mathcal{T})$.
\end{itemize}
\end{definition}

The family of all unconditionally closed subsets of $G$ coincides with the family of closed sets of a $T_1$ topology ${\mathfrak M}_G$ on $G$, namely the infimum (taken in the lattice of all topologies on $G$) 
of all Hausdorff group topologies on $G$. This topology has been introduced in \cite{DS_OPIT, DS_JGT} as the {\em  Markov topology} of $G$.  

Recall that a Hausdorff group topology $\mathcal{T}$ on a group $G$  is called {\em precompact\/}  (or {\em totally bounded\/}) provided that $(G,\mathcal{T})$ is
 (isomorphic to)  a subgroup of some compact Hausdorff  group or, equivalently, if the completion of $(G,\mathcal{T})$ with respect to the two-sided uniformity is compact. Let ${\mathfrak P}_G$ be the infimum  of all precompact  Hausdorff group topologies on $G$.  Clearly, ${\mathfrak P}_G$ is a $T_1$ topology on $G$, which we  call the {\em precompact Markov topology\/} of $G$ \cite{DS_OPIT}.

One can easily see that the family of all algebraic subsets of $G$ is closed under finite unions and arbitrary intersections, and  contains $G$ and all finite subsets of
$G$; thus, it  can be taken as the family of closed sets of a unique $T_1$ topology ${\mathfrak Z}_G$ on $G$. Markov \cite{Mar1, Mar} defined  the {\em algebraic closure\/} of a subset $X$ of a group $G$ as the intersection of all algebraic subsets of $G$ containing $X$, i.e., the smallest algebraic set that contains $X$.
This definition satisfies the conditions necessary for introducing
a topological closure operator on $G$. Since a topology on a set is uniquely determined by its closure operator, it is fair to say that Markov was the first to
(implicitly) define  the topology ${\mathfrak Z}_G$, though he did not  name it. To the best of our knowledge, the first name for this topology appeared explicitly in print in a 1977 paper by Bryant \cite{Bryant}, who called it a {\em verbal topology\/} of $G$. In a more recent  series of papers  beginning with \cite{BMR}, Baumslag, Myasnikov and Remeslennikov have  developed algebraic geometry over an abstract group $G$.  In an analogy with the celebrated 
Zariski topology from algebraic geometry, they introduced the {\em Zariski topology\/} on the finite powers $G^n$ of a group $G$. In the particular case  when
$n=1$, this topology coincides with the verbal topology of Bryant. For this reason, the topology $\Zar_G$ is also called the {\em Zariski topology\/} of $G$ in \cite{DS_OPIT, DS_JGT}.  

Note that $(G,{\mathfrak Z}_G)$, $(G,{\mathfrak M}_G)$ and  $(G,{\mathfrak P}_G)$ are quasi-topological groups, i.e.,  their  inversion and shifts are continuous; see \cite{AT}.

\begin{fact}
\label{three:topologies}
 ${\mathfrak Z}_G\subseteq {\mathfrak M}_G\subseteq {\mathfrak P}_G$ for every group $G$.
\end{fact}

\begin{proof}
An elementary algebraic subset of $G$ must be closed in every Hausdorff group topology on $G$, which gives the first inclusion. The second inclusion is obvious.
\end{proof}

\subsection{Markov's ``algebraic  versus unconditionally closed'' problem} 

In 1944, Markov \cite{Mar1} (see also \cite{Mar}) posed his celebrated  problem: {\em is  every unconditionally closed subset of
a group algebraic\/}? Using the language of Markov and Zariski topologies, this question can be  naturally reformulated as the problem  of coincidence of these  topologies: {\em does the equality ${\mathfrak Z}_G = {\mathfrak M}_G$ hold for every group $G$\/}?  Markov himself obtained a positive answer in the case when $G$ is countable \cite{Mar}. Moreover, in 
 \cite{Mar}, Markov attributes to Perel'man the fact that ${\mathfrak Z}_G = {\mathfrak M}_G$ for every abelian group $G$.  To the best of our knowledge,  the proof of this fact  never appeared  in print until \cite{DS_JGT}. In the present  manuscript,  we further   strengthen this result from  \cite{DS_JGT} as follows:

\begin{thma} 
${\mathfrak Z}_G={\mathfrak M}_G={\mathfrak P}_G$ for every abelian group $G$.
\end{thma}

Reformulating this theorem in Markov's terminology results in the following Corollary B.

\begin{thmb}
A subset of an abelian group $G$ is algebraic if (and only if) it is closed in every precompact Hausdorff group topology on $G$.
\end{thmb}

An example of a group $G$ with ${\mathfrak Z}_G\not={\mathfrak M}_G$ was found by Hesse \cite{Hesse}, who apparently was unaware that his results solve Markov's problem in the negative. This problem was later  highlighted as an open problem in the survey \cite{CHR}, and an example of a group $G$ with ${\mathfrak Z}_G\not={\mathfrak M}_G$ under the Continuum Hypothesis CH was   recently  provided in \cite{S1}.

If the group $G$ is infinite, the topology ${\mathfrak M}_G$ is discrete if and only if $G$ is  {\em non-topologizable\/}, i.e., it does not admit a non-discrete Hausdorff group topology.
The existence of non-topologizable groups has been another long-standing open problem of Markov \cite{Mar}; it was later resolved positively  through an example under CH 
by Shelah \cite{Shelah}, an uncountable ZFC example by Hesse \cite{Hesse} and a countable  ZFC example by Ol$'$shanskij \cite{O}. 

Observe that the topology ${\mathfrak P}_G$ is discrete if and only if the group $G$ is not {\em maximally almost periodic\/}, i.e., it does not admit a precompact Hausdorff group topology. A classical example of a group $G$ with discrete ${\mathfrak P}_G$ is the group $SL(2,\mathbb{C})$ of all complex $2\times2$ matrices with determinant $1$ \cite{vNW}.
Since this $G$ is topologizable (by its usual topology), one has  ${\mathfrak M}_G\not={\mathfrak P}_G$.

In view of Fact \ref{three:topologies}, if a group $G$ has discrete $\Zar_G$, then ${\mathfrak Z}_G={\mathfrak M}_G={\mathfrak P}_G$ holds because all three topologies become discrete. Such groups are extremely rare, but a variety of examples have been  constructed by Ol$'$shanskij and his school:

\begin{example}\label{discrete:Zariski} 
{\rm 
 \begin{itemize} 
 \item[(a)] Ol$'$shanskij's example \cite{O} of a countable non-topologizable group has discrete $\Zar_G$.
   \item[(b)] Klyachko and Trofimov \cite{KT} constructed a finitely generated torsion-free group $G$ such that $\Zar_G$ is discrete.
   \item[(c)] Trofimov \cite{T} proved that every group $H$ admits an embedding into a group $G$ with discrete $\Zar_G$. 
   \item[(d)] Morris and Obraztsov \cite{MO} modified Ol$'$shanskij's example from item (a) to build, for any sufficiently large prime $p$, a continuum of pairwise non-isomorphic infinite non-topologizable groups of exponent $p\sp 2$, all of  whose proper subgroups are cyclic. 
 \end{itemize}
}
 \end{example}

\subsection{Realization of the Zariski closure by some Hausdorff group topology} 

Given a topological space $(Y,\mathcal{T})$, we  denote the $\mathcal{T}$-closure of a set $X\subseteq G$  by $\CL_{\mathcal{T}}(X)$. 

 Let $X$ be a subset of a group $G$. Given a Hausdorff group topology $\mathcal{T}$ on $G$,  one has  $\Zar_G\subseteq  \Mar_G\subseteq \mathcal{T}$, and therefore, $\CL_{\mathcal{T}}(X)\subseteq \CL_{\Mar_G}(X)\subseteq \CL_{\Zar_G}(X)$. This chain of  inclusions naturally leads to the following problem: can one always find a Hausdorff group topology $\mathcal{T}$ on $G$  such that $\CL_{\mathcal{T}}(X)= \CL_{\Mar_G}(X)= \CL_{\Zar_G}(X)$? We shall call this question the {\em realization problem for the 
 Zariski closure\/}; see \cite{DS_OPIT}. This problem was first considered  by Markov in \cite{Mar}, who proved  that for every subset $X$ of a countable group  $G$, there exists a metric group topology $\mathcal{T}$ on $G$ such that $\CL_{\mathcal{T}}(X)= \CL_{\Zar_G}(X)$. We make the following contribution to this general problem  in the abelian case:

\begin{thmc} For an abelian group $G$, the following conditions are equivalent:
\begin{itemize}
  \item[(i)] $|G|\le \cont$,
  \item[(ii)] for every subset $X$ of $G$, there exists a precompact metric group topology $\mathcal{T}_{{X}}$ on $G$ such that $\CL_{\mathcal{T}_{{X}}}(X)=\CL_{\Zar_G}(X)$,
  \item[(iii)] for every countable family $\mathscr{X}$ of subsets of $G$, one can find a precompact metric group topology $\mathcal{T}_{\mathscr{X}}$ on $G$ such that $\CL_{\mathcal{T}_{\mathscr{X}}}(X)=\CL_{\Zar_G}(X)$ for every $X\in\mathscr{X}$.
\end{itemize}
\end{thmc}

It should be noted that item (iii) cannot be pushed further to accommodate families of size $\omega_1$; see Remark \ref{closure:of:many:sets:cannot:be:realized}.

A counterpart to this theorem, with the word "metric" removed from both items (ii) and (iii) and the inequality in item (i) relaxed to $|G|\le 2^\cont$, is proved in our subsequent paper \cite{DS_Kronecker}.

\subsection{Characterization of potentially dense subsets  of abelian groups of size at most $\cont$}

The last section of Markov's paper \cite{Mar} is exclusively dedicated to the following problem: {\em which subsets of a group $G$ are potentially dense in $G$?\/} Markov succeeded in proving that every infinite subset of $\Z$ is potentially dense in $\Z$ \cite{Mar}. This was strengthened in \cite[Lemma 5.2]{DT0} by proving that every infinite subset of $\Z$ is dense in some precompact metric group topology on $\Z$.  (The authors of \cite{Mar} and \cite{DT0} were  apparently unaware that both these results easily follow from  the uniform distribution theorem of Weyl \cite{W}.) Further progress was  made by Tkachenko and Yaschenko \cite{TY}, who proved the following theorem: if an abelian group $G$ of size at most $\cont$  is either almost  torsion-free or has exponent $p$ for some prime $p$, then every infinite subset of $G$ is potentially dense in $G$. (According to \cite{TY}, an abelian group $G$ is {\em almost torsion-free\/} if $r_p(G)$ is finite for every prime $p$.)  In this manuscript, we obtain the following complete characterization of potentially dense subsets of abelian groups of size at most $\cont$:

\begin{thmd} Let $X$ be a subset of an abelian group $G$ such that $|G|\le\cont$. Then the following conditions are equivalent:
\begin{itemize}
 \item[(i)] $X$ is potentially dense in $G$,
 \item[(ii)] $X$ is $\mathcal{T}$-dense in $G$ for some precompact metric group topology on $G$,
 \item[(iii)] $\CL_{\Zar_G}(X)=G$.
\end{itemize}
\end{thmd}

Note that a precompact metric group has size at most $\cont$, so item (ii) of Theorem D implies $|G|\le \cont$. Therefore, this cardinality restriction in Theorem D is necessary. A counterpart of this  theorem, with the word "metric" removed from  item (ii) and  the condition on a group $G$  relaxed to $|G|\le 2^\cont$, is proved in our  forthcoming
paper \cite{DS_Kronecker}.

Item (iii) of Theorem D highlights the importance of characterizing Zariski dense subsets of an abelian group $G$. This is accomplished in Theorem \ref{characterizing:Zariski:dense:sets}. Moreover, we explicitly calculate the Zariski closure of an arbitrary subset of $G$; see Theorem \ref{computing:the:closure} and its corollaries in Section \ref{Closure:section}.
 
\subsection{The structure of the Zariski topology $\Zar_G$ of an abelian group}

Bryant \cite{Bryant} established that $\Zar_G$ is Noetherian for an abelian-by-finite group $G$; for the convenience of the reader, we  provide a self-contained proof of this fact in the abelian  case in Theorem \ref{Zariski.topology.is.Noetherian}. In particular, for an abelian group $G$, the space $(G,\Zar_G)$ has common properties shared by all Noetherian ($T_1$) spaces; see Facts \ref{Noetherian:facts} and  \ref{decomposition:into:irreducible:components}. We prove that, for an abelian group $G$, the space $(G,\Zar_G)$ also has some specific properties not shared by all Noetherian spaces:

\begin{itemize}
\item[($\alpha$)]
Every subspace $X$ of $(G,\Zar_G)$ contains a countable subset $Y$ that is dense  in $X$, i.e., $(G,\Zar_G)$ is {\em hereditarily separable\/},  using topological terminology; see \cite{Eng}.
\item[($\beta$)]
The subset $Y$ of $X$ from item ($\alpha$) can be chosen  so that it has the cofinite topology; it now follows that for every  point $x$ in $\Zar_G$-closure of $X$ one can find a sequence $\{y_n:n\in\mathbb{N}\}$ of points of $Y$ such that the sequence $y_n$ converges to $x$ in  the Zariski topology $\Zar_G$.  This last property implies that $(G,\Zar_G)$ is {\em  Fr\'echet-Urysohn\/},  using topological terminology; see \cite{Eng}.
\item[($\gamma$)]
Irreducible components of an {\em elementary\/} algebraic set are disjoint, but irreducible components of an   algebraic set need not be disjoint.
\end{itemize}

Item ($\alpha$) is proved in Corollary \ref{hereditarily:separable}, item ($\beta$) is proved in Corollary \ref{Frechet-Urysohn}, and item ($\gamma$) follows from Corollary \ref{connected:vs:irreducible}(ii) and Example \ref{non-disjoint:components}.

\subsection{A brief overview of the structure of the paper}

The paper is organized as follows. Since the Zariski topology $\Zar_G$ of an abelian group $G$ is Noetherian, we recall basic properties of Noetherian spaces in Section \ref{section:2}. 
Irreducible sets play a principal role in the geometry of Noetherian spaces, as witnessed by the fact that every subset $X$ of a Noetherian space is a union of a uniquely determined finite family of irreducible sets;  these irreducible sets are the so-called irreducible components
(that is, the maximal irreducible subsets) of $X$. At the end of Section \ref{section:2}, we recall  also  relevant facts about combinatorial dimension $\dim$.
 
Section \ref{section:3} introduces the Zariski topology $\Zar_G$ and establishes its basic properties. For every abelian group $G$, the pair $(G,\Zar_G)$ is a so-called quasi-topological group, that is, a $T_1$-group in the terminology of Kaplansky \cite{K},  but  it is not a topological group unless $G$ is finite; see Corollary \ref{Zariski:topology:never:group}.
Motivated by items ($\alpha$) and ($\beta$) above, we call a countably infinite subset $X$ of an abelian group $G$ a {\em \Zc{\Zar_G}} if $X$ has the cofinite 
topology (that is, the coarsest $T_1$ topology) as a subspace of $(G,\Zar_G)$;  see Definition \ref{definition:of:Zariski:atom}. An important result demonstrating why 
\Zc{\Zar_G}s are so useful is Proposition \ref{general:properties:of:atoms} which says, in particular, that the Zariski closure of a \Zc{\Zar_G}\ is an irreducible (and thus,  an elementary algebraic) set. Furthermore, \Zc{\Zar_G}s are precisely the irreducible one-dimensional countably infinite subsets of $(G,\Zar_G)$; see Fact \ref{FactZcurve}.  

Since the topology of any subset of $(G,\Zar_G)$ is completely determined by its finitely many irreducible components, 
we study in detail Zariski irreducible sets in Section \ref{section:4}.  In Theorem \ref{essential:order:and:components}(ii), we prove that the irreducible component of zero and the connected component of zero of  the space $(G,\Zar_G)$ are both equal to $G[n]$, where $n$  coincides with the so-called essential order $eo(G)$ of the group $G$   (defined 
by Givens and Kunen \cite{GK} in the case of a bounded group $G$; see  Definition \ref{definition:of:essential:order}). 

Section \ref{section:5} provides the ``technical core'' of the  manuscript. A prominent role in our paper is played by the notion of an \round{n} set; see Definition \ref{def:of:almost:n:torsion:sets}. To  assist the reader in better understanding of this notion,   we provide its equivalent forms in Lemma \ref{equivalent:condition:for:round:sets}. The equivalent condition for $n=0$ is particularly clear:  A subset of an abelian group $G$ is \round{0} if and only if its intersection  with every coset of the form $a+G[m]$, where $a\in G$ and $m\in\mathbb{N}\setminus\{0\}$, is finite; see Corollary \ref{0-round:vs:n(-)=0}.
The general case is similar to this special case, except that the finiteness condition is imposed only on integers $m$ that are proper divisors of $n$; see Lemma \ref{equivalent:condition:for:round:sets}(ii). Proposition \ref{round:sets:as:dense:atoms} uncovers a connection between the algebraic notion of an \round{n} set and the topological notion of a \Zc{\Zar_G}. The main result of Section \ref{section:5} is Proposition \ref{finding:n-round:sets} characterizing  sets that contain \round{n} subsets.
An especially simple characterization of sets containing an \round{0} set is  given in Proposition \ref{sets:that:contain:zero-round:sets}. Corollary \ref{description:of:when:there:are:n-round:sets}, Corollary  \ref{groups:that:conatin:0-round:sets} and  Proposition \ref{almost:torsion-free:groups:by:means:of:n-round}
describe abelian groups that do, or do not, contain \round{n} sets, for different integers $n$. 

Section \ref{section:6} reveals the main reason for considering \round{n} sets in this manuscript.  Indeed, in Theorem \ref{corollary:about:atoms:being:translates:of:round:sets} 
we prove that every \Zc{\Zar_G} $X$ is a translate of some  \round{n} set, for a suitable integer $n$ that is {\em uniquely determined\/} by $X$.
This creates an essential bridge between algebra (\round{n} sets) and topology (\Zc{\Zar_G}s).   Another crucial result in this section is Corollary \ref{closure:irreducible:sets} saying that Zariski irreducible sets are precisely those  sets that contain a $\Zar_G$-dense \Zc{\Zar_G}
 (or, equivalently, those sets that contain a $\Zar_G$-dense translate of an \round{n} set, for a suitable integer $n$). In particular, 
Zariski closures of \Zc{\Zar_G}s give all closed irreducible subsets of $(G,\Zar_G)$; see Corollary \ref{closures:of:atoms}. Since every subspace $X$ of $(G,\Zar_G)$ is a finite union of its irreducible components, this shows that \Zc{\Zar_G}s (that is, translates of \round{n} sets) completely determine the topology of $X$.

Since a potentially dense set is Zariski dense,  as a necessary step towards solving  Markov's potential density problem, we study Zariski dense sets in Section \ref{section:7}. Theorem \ref{characterizing:Zariski:dense:sets} completely describes such sets by means of (translates of) \round{n} sets. Zariski dense subsets of unbounded abelian groups are given an especially nice and simple characterization in Theorem \ref{Zariski:dense:sets:in:unbounded:groups}.

Building on results from previous sections, in Section \ref{Closure:section} we  derive a complete description of the Zariski closure of an arbitrary set via \Zc{\Zar_G}s or, equivalently,  translates of \round{n} sets.  

    Theorem C is proven in Section \ref{realizing:closure:section}, while Theorem A is proven in Section \ref{section:10}. Finally, Theorem D is proven in Section \ref{section:11}.  In Section \ref{open:questions} we collect some open questions about
 the Zariski and Markov topologies in the non-abelian case, based on  the results  in the abelian case. 
 
\section{Background on Noetherian spaces, irreducible components and combinatorial dimension $\dim$}
\label{section:2}

Recall that a topological space is said to be {\em Noetherian\/} if it satisfies the ascending chain condition on open sets,  or the equivalent descending chain condition on closed sets. Kaplansky \cite[Chap. IV, p. 26]{K} calls the  Noetherian $T_1$-spaces {\em $Z$-spaces}. We summarize here the key (mostly well-known) properties of Noetherian spaces.

\begin{fact}
\label{Noetherian:facts}
\begin{itemize}
  \item[(1)] A subspace of a Noetherian space is Noetherian.
  \item[(2)] If $f:X\to Z$ is a continuous surjection and $X$ is Noetherian, then so is $Z$.
  \item[(3)] Every (subspace of a) Noetherian space is compact.
  \item[(4)] Every non-empty family of closed subsets of a Noetherian space has a minimal element under  set inclusion.
  \item[(5)] Every infinite  subspace $Y$ of a Noetherian space $X$ contains an infinite subspace $Z$ such that every proper closed subset of  $Z$  is finite.
  \item[(6)] No infinite subspace of a Noetherian space can be Hausdorff; in particular, every Hausdorff Noetherian space is finite.
\end{itemize}
\end{fact}
\begin{proof} (1) and (2) are straightforward. 

In view of (1), to show (3)  it suffices to prove that every Noetherian space is compact. This immediately follows from the fact that every descending chain of closed sets 
stabilizes.

(4) follows easily from the fact that the Noetherian spaces satisfy the descending chain condition on closed sets. 

To prove (5), it suffices  to consider the case when $Y=X$, since according to (1), a subspace of a Noetherian space is Noetherian.  Let  $\mathcal{F}$ be the family of all infinite closed subsets of $X$. Since $X\in\mathcal{F}\not=\emptyset$, we can use item (4) to find a minimal element $Z$ of $\mathcal{F}$. Clearly, every proper closed subset of $Z$ is finite.

(6) Let $X$ be an infinite Noetherian space. Apply item (5) to $Y=X$ to get an infinite subspace $Z$ of $Y=X$ as in the conclusion  of this item. Since $Z$ is an infinite space without proper infinite closed subsets, $Z$ cannot be Hausdorff. Since $Z$ is a subspace of $X$, we conclude that $X$ is not Hausdorff either. 
\end{proof}

Bryant \cite{Bryant} discovered a useful technique for building Noetherian spaces: 

\begin{fact}
\label{building:Noetherian:spaces}
{\rm (\cite{Bryant})}
Let $\mathscr{E}$ be a family of subsets of a set $X$ closed under finite intersections and satisfying the descending chain condition. Assume also that $X\in \mathscr{E}$. Then the family $\fin{\mathscr{E}}$ consisting of finite unions of the members of $\mathscr{E}$ forms the family of closed sets of a unique topology $\mathcal{T}_{\mathscr{E}}$ on 
$X$ such that the space $(X, \mathcal{T}_{\mathscr{E}})$ is a Noetherian space.
\end{fact}

From this fact, we can easily show that finite products of Noetherian spaces are Noetherian.

Recall that a topological space $X$ is called {\em irreducible\/}  ({\em connected\/})  provided that for every partition  (respectively, disjoint partition) $X=F_0\cup F_1$ of $X$ into closed sets $F_0$ and $F_1$, either $F_0=X$ or $F_1=X$ holds. Note that  an irreducible space is connected. For an example of a connected $T_1$-space that is not irreducible, take the reals or any connected infinite Hausdorff space. 

An easy  induction establishes the following fact.

\begin{fact}
\label{fact:on:finite:unions}
If $X\not=\emptyset$ is an irreducible subset of a space $Y$ and $\mathcal{F}$ is a finite family of closed subsets of $Y$ such that $X\subseteq \bigcup\mathcal{F}$, then $X\subseteq F$ for some $F\in\mathcal{F}$.
\end{fact}

We omit the easy proof of the following fact. 

\begin{fact}\label{Fact:irreducible}
\begin{itemize}
\item[(i)] A space $X$ is irreducible if and only if every non-empty open subset of $X$ is dense in $X$.
\item[(ii)] Every continuous map from an irreducible space to a Hausdorff space  is constant. In particular, the only irreducible Hausdorff spaces are the singletons.
\item[(iii)] Every space with a dense irreducible subspace is irreducible. 
\item[(iv)] A dense subset $Y$ of a topological space $X$ is irreducible if and only if $X$ is irreducible. 
\end{itemize}
 \end{fact}

If $X$ is a space and $x\in X$, then a maximal element of the family of all connected (irreducible) subsets of $X$ containing $x$,  ordered by set-inclusion, is called the {\em connected component\/}  (respectively, an {\em irreducible component\/}) of $x$. Note that the connected component of a point is uniquely determined, while the irreducible ones need not be unique. Connected and irreducible components are always closed. 

\begin{fact}
\label{decomposition:into:irreducible:components}
{\rm (\cite{Ha})}
 Let $Y$ be a Noetherian space. Then every  subset $X$ of $Y$ admits a unique decomposition  $X=X_0\cup\dots\cup X_n$ into a finite union of  irreducible, relatively closed
  (in the subspace topology of $X$) subsets $X_i$ such that  $X_i\setminus X_j\not=\emptyset$  for $i\not = j$. Moreover, each such $X_i$ is an irreducible component of $X$. 
\end{fact}
\begin{proof}
Assume that $X$ is not a finite union of closed irreducible subspaces. Then the family $\mathcal Z$ of closed subspaces $Z$ of $X$ that are not a finite union of closed irreducible subspaces is non-empty. Choose a minimal element $Z$ of  $\mathcal Z$.  Then $Z$ cannot be irreducible, so $Z=F_0 \cup F_1$, where $F_0$ and $F_1$ are proper closed subspaces of $Z$. Then, by  the minimality of $Z$, both $F_0$ and $F_1$ are finite unions of closed irreducible subspaces of $X$, which yields $Z\not \in \mathcal Z$, giving a contradiction.  This proves the existence of  a decomposition $X=X_0\cup\dots\cup X_n$ of $X$  into a finite union of closed irreducible subsets. By deleting some 
$X_i$, we can assume, without loss of generality, that   $X_i\setminus X_j\not=\emptyset$ for $i\not = j$.

To prove the uniqueness of the decomposition, suppose that $X=X'_0\cup\dots\cup X'_m$ is another decomposition of $X$ into a finite union of closed irreducible subsets such that $X'_i\setminus X'_j\not=\emptyset$ for $i\not = j$.  Fix $i=1,2,\ldots, n$. From $X_i\subseteq X'_0\cup\dots\cup X'_m$ and  Fact \ref{fact:on:finite:unions}, we deduce that $X_i\subseteq X'_j$ for some $j$. Analogously,  from $X'_j \subseteq X_0\cup\dots\cup X_m$ and  Fact \ref{fact:on:finite:unions}, we must have $X'_j \subseteq X_k$ for some $k$. Now $X_i \subseteq X_k$ yields $k=i$, and hence  $X_i =X'_j$. This shows that $m=n$, and thus, $X_k=X'_{f(k)}$ for all $k=0,1,\dots,n$, where $f$ is an appropriate permutation of $n$. 

It remains only to be shown that each $X_i$ is a maximal irreducible subset of $X$. Let $T$ be an irreducible subset of $X$ such that $X_i\subseteq T$. Then $T\subseteq X= X_0\cup\dots\cup X_n$, and Fact \ref{fact:on:finite:unions} yields that $X_i\subseteq T\subseteq X_j$ for some $j$. This implies $i=j$ and $T=X_i$.
\end{proof}

\begin{fact}\label{remark:irreducible:components} Let $X$ be a subset of a Noetherian space $Y$.
\begin{itemize}
\item[(i)] $X$ has finitely many connected components.  Every connected component $C_k$ of $X$ is a clopen subset of $X$, and the finite family of connected components of $X$ forms a (disjoint) partition $X=\bigcup_{k=1}^m C_k$. 
\item[(ii)]
Every irreducible component $X_i$ of $X$ is contained in some connected component $C_k$. 
\end{itemize}
Moreover, if $Y$ is a $T_1$-space, then the following also holds:
\begin{itemize}
\item[(iii)] The set $D$ of isolated points of $X$ is finite,  and $\{d\}$ is a connected component of $X$ for every $d\in D$.
\item[(iv)]
If $X$ is infinite, then the set $X\setminus  D$ has no isolated points, its irreducible components are infinite and coincide with the infinite irreducible components of $X$. 
\end{itemize}
\end{fact}

\begin{proof} 
(i)
Assume that $X$ has infinitely many connected components $\{C_n:n\in\mathbb{N}\}$. Choose a point $y_n\in C_n$ for every $n\in\mathbb{N}$.  
By Fact \ref{Noetherian:facts}(5), there exists an infinite subspace $Z$ of the set $Y=\{y_n:n\in\mathbb{N}\}$ such that all proper closed subsets of $Z$ are finite.
Clearly, $Z$ is connected. Since $Z$  intersects with (infinitely many) distinct connected components of $Y$, we get a contradiction. Therefore, $X$ has finitely many distinct connected components $C_1,\dots,C_m$.

(ii)
Let $X_i$ be an irreducible component of $X$. Since $X_i\subseteq X=\bigcup_{k=1}^m C_k$, from Fact \ref{fact:on:finite:unions} we conclude that $X_i\subseteq C_k$ for some $k=1,\dots,m$.

(iii) Let $d$ be an isolated point of $X$. Since $Y$ is a $T_1$-space, the set $\{d\}$ is both open and closed in $X$. It follows that $\{d\}$ is a connected component of $X$.
Applying (i), we conclude that $D$ is finite.

(iv) Suppose now that $X$ is infinite. Since $Y$ is a $T_1$-space, and $D$ is finite by (iii), $X\setminus D$ is a non-empty open subset of $X$. Therefore, an isolated point of $X\setminus D$ is an isolated point of $X$ as well,  contradicting the definition of $D$. Let $X\setminus D=\bigcup_{i=1}^n X_i$ be the decomposition of $X\setminus D$ into irreducible components of $X\setminus D$. Now one can easily see that  $X=\bigcup_{i=1}^n X_i\cup\bigcup_{d\in D} \{d\}$ is the decomposition of $X$ into irreducible components of $X$.
\end{proof}

A space in which all closed irreducible subsets have a dense singleton is known as a {\em sober} space. Recall that the spectrum of every commutative ring is a sober space. 

\begin{fact}
\label{sober:fact}
A sober Noetherian $T_1$-space must be finite.
\end{fact}

\begin{proof}
Let $X$ be a sober Noetherian $T_1$-space, and let $X=X_0\cup \dots\cup X_n$ be the decomposition of $X$ into its irreducible components; see Fact \ref{decomposition:into:irreducible:components}. Since $X$ is sober, each $X_i$ contains an element $x_i$ such that $\{x_i\}$ is dense in $X_i$. Since $X$ is a $T_1$-space, 
the set $\{x_i\}$ is closed in $X$, which implies $X_i=\{x_i\}$. Therefore, $X=\{x_1,\dots,x_n\}$.
\end{proof}

Given a topological space $X$ and a natural number $n$, we write $\dim X\geq n$ if there exists a strictly increasing chain 
\begin{equation}
\label{eq.no.1}
F_0\subseteq  F_1\subseteq  \ldots \subseteq  F_n 
\end{equation} 
of non-empty irreducible closed subsets of $X$. The {\em combinatorial dimension\/} $\dim X$ of a space $X$ is the smallest  number $n\in\mathbb{N}$ satisfying $\dim X\le n$, if such a number exists, or $\infty$ otherwise. Clearly, every Hausdorff space (as well as every anti-discrete space) has combinatorial dimension 0.  Recall that the Krull dimension of a commutative ring coincides with the combinatorial dimension of its spectrum. 

\begin{fact}\label{FactDim}
Let $X$ be a Noetherian $T_1$-space. 
\begin{itemize}
  \item[(a)] $\dim X>0$ if and only if $X$ is infinite.
  \item[(b)] If  $Y\subseteq X$, then $\dim Y \leq \dim X$.  
  \item[(c)] If $S=\bigcup_{j=1}^k S_j$ is a decomposition of a subset $S$ of $X$ into irreducible components of $S$, then $\dim S=\max_{1\le j\le k}\dim S_j$.
  \item[(d)]  If $Y,Z\subseteq X$, then $\dim (Y\cup Z)=\max \{\dim Y, \dim Z\}$.  
\end{itemize}
\end{fact}

\begin{proof} (a) follows directly from the definition.

(b) If $Y$ is closed, then every chain (\ref{eq.no.1}) witnessing $\dim Y \geq n$ will witness $\dim X \geq n$ as well. If $Y$ is dense in $X$, then 
every chain (\ref{eq.no.1}) witnessing $\dim Y \geq n$ will give rise to a chain $\overline{F_0}\subseteq  \overline{F_1}\subseteq  \ldots \subseteq  \overline{F_n}$ witnessing $\dim X \geq n$.  In the general case, we have $\dim Y \leq \dim \overline{Y} \leq \dim X$ by the preceding two cases. 

(c) Indeed, the inequality $\dim S\geq \max_{1\le j \le k}\dim S_j$ follows from (b). Let $n = \dim S $, and assume that
(\ref{eq.no.1}) is a chain of closed irreducible subsets of $S$ witnessing $\dim S\ge n$. Since the irreducible set $F_n$ is contained in the finite union $\bigcup_{j=1}^k S_j=S$ of closed sets, we conclude that $F_n \subseteq S_i$ for some $i = 1,2,\ldots k$. In other words, the whole chain
(\ref{eq.no.1}) is contained in $S_i$, implying $\dim S_i\geq n$.  Hence  $\dim S=n\le \dim S_i\le \max_{1\le j\le k}\dim S_j$. This proves (c).
 
(d)
 Let $Y=\bigcup_{i=1}^n C_i$ and $Z=\bigcup_{j=1}^m K_j$  be decompositions of $Y$ and $Z$ into irreducible components;
see Fact \ref{decomposition:into:irreducible:components}.  Note that $Y \cup Z= \bigcup_{i=1}^n C_i\cup \bigcup_{j=1}^m K_j$ need not be the representation of $Y \cup Z$ as a union of irreducible components, since $C_i\cup K_j=C_i$ or $C_i\cup K_j= K_j$ (i.e., $K_j \subseteq C_i$ or $C_i \subseteq K_j$) may occur.  Nevertheless, by removing the redundant members  in the union $\bigcup_{j=1}^m K_j$, one obtains the   decomposition $Y \cup Z=\bigcup _\nu L_\nu$  of $Y\cup Z$ into irreducible components $L_\nu$, and from  (c) we get
$$
\dim (Y\cup Z)=\max\{\dim L_\nu\}\le \max\left\{\max_{1\le i\le n}\dim C_i, \max_{1\le j\le m}\dim K_j\right\}=\max \{\dim Y, \dim Z\}.
$$
The inverse inequality $\dim (Y\cup Z)\ge \max \{\dim Y, \dim Z\}$ follows from  (b).
\end{proof}

\begin{fact}\label{FactZcurve} 
For a subspace $X$ of a Noetherian space $Y$, the following conditions are equivalent:
\begin{itemize} 
  \item[(a)] $X$ is an irreducible $T_1$-space of combinatorial dimension 1,
  \item[(b)] $X$ is infinite and carries  the cofinite topology $\{X\setminus F: F$ is a finite subset of $X\}\cup \{\emptyset\}$,
  \item[(c)] $X$ is an infinite space with the coarsest $T_1$-topology.
\end{itemize}
\end{fact}

\begin{proof} It is easy to see that (b) and (c) are equivalent and imply (a), even without the assumption that $Y$ is Noetherian.

(a)~$\to$~(b) Since $\dim X=1>0$, our $X$ must be infinite by Fact \ref{FactDim}(a). Let $F$ be a non-empty closed subset of $X$, and let $F=F_0\cup\dots\cup F_n$ be the decomposition of $F$ into its irreducible components $F_i$; see Fact \ref{decomposition:into:irreducible:components}.
Let $i\le n$ be a non-negative integer. Choose $x_i\in F_i$. Since all three sets in the chain $\{x_i\}\subseteq F_i\subseteq X$ are irreducible and $\dim X=1$, at least one of the inclusions cannot be proper; that is, either  $F_i=X$ or $F_i=\{x_i\}$. From $F=F_0\cup\dots\cup F_n$, we conclude that either $F=X$ or $F$ is finite. 
\end{proof}

Easy examples show that the implication (a)~$\to$~(b) fails if $X$ is not Noetherian. 

\section{Properties of  algebraic sets and the Zariski topology of an abelian group} 
\label{section:3}

For an abelian group $G$ and an integer $n$, we set $G[n]=\{x\in G: nx=0\}$. Clearly, $G[n]$ is a subgroup of $G$, with $G[0]=G$ and  $G[1]=\{0\}$. 

As usual, for integers $m$ and $n$, $n|m$ means that $n$ is a divisor of $m$, and  $(m,n)$ denotes the greatest common divisor of $m$ and $n$, in case at least one of these integers is non-zero. For the sake of convenience, we set $(0,0) =0$. 

\begin{lemma}
\label{comparing:elementary:algebraic:sets}
\label{intersection:of:two:elementary:algebraic:sets}
 Suppose $G$ is an abelian group, $a,b\in G$ and $n,m \in \mathbb{N}$. Then:
\begin{itemize}
  \item[(i)] $a+G[n]\subseteq b+G[m]$ if and only if $G[n]\subseteq G[m]$ and $a-b\in G[m]$,
  \item[(ii)] $a+G[n]=b+G[m]$ if and only if $G[n]= G[m]$ and $a-b\in G[m]=G[n]$,
  \item[(iii)]  $G[n]\cap  G[m]=G[d]$ for $d=(m,n)$; in particular, $G[n]=G[d]$ provided that $G[n]\subseteq G[m]$,
  \item[(iv)] if $z_0\in  (a+G[n])\cap(b+G[m])$, then $(a+G[n])\cap(b+G[m])=z_0+G[d]$, where $d=(n,m)$.
  \end{itemize}
\end{lemma}

\begin{proof} (i) From $a+G[n]\subseteq b+G[m]$, we deduce $G[n]\subseteq b-a+G[m]$, so $b-a\in G[m]$ and $G[n]\subseteq G[m]$. 

(ii) follows from (i), and  the proof of (iii) is  straightforward.

(iv) By our hypothesis,  $z_0-a\in G[n]$ and $z_0-b\in G[m]$, so  we conclude that $z_0+G[n]=a+G[n]$ and $z_0+G[m]=b+G[m]$.  Since $G[n]\cap G[m]= G[d]$ by (iii), we get
$$
(a+G[n])\cap(b+G[m])=(z_0+G[n])\cap (z_0+G[m])=z_0+(G[n]\cap G[m]) = z_0+G[d].
$$
\end{proof}

\begin{notation}\label{Def:elem:alg:set} 
{\rm 
Let $G$ be an abelian group. Then $\EE_G$  denotes the family of all elementary algebraic sets of $G$, and  ${\mathfrak A}_G$  denotes the family of all finite unions of elementary algebraic sets of $G$. For convenience, we define $\EE_{\{0\}}=\{\emptyset, \{0\}\}$, even though $\emptyset$ is not an elementary algebraic subset of the trivial group $G=\{0\}$.
}
\end{notation}

The next lemma collects some basic properties of $\EE_G$.

\begin{lemma}\label{lemma:elementary:algebraic:sets}
For every abelian group $G$, the family $\EE_G$ has the following properties:
\begin{itemize}
\item[(i)] if $E\not=\emptyset$, then $E\in \EE_G$ if and only if $E=a+G[n]$ for some $a\in G$ and $n\in \mathbb{N}$, 
\item[(ii)] $\emptyset\in \EE_G$ and $G\in \EE_G$,
\item[(iii)] $\EE_G$ is closed  under the operation $x\mapsto -x$,
\item[(iv)] $\EE_G$ is closed under taking translations,
\item[(v)]  $E_1, \ldots, E_n\in \EE_G$ implies $E_1+\ldots + E_n\in \EE_G$,
\item[(vi)] $\EE_G$ is closed with respect to taking finite intersections,
\item[(vii)] every descending chain in ${\mathfrak E}_G$ stabilizes, so  $\EE_G$ is closed with respect to  taking arbitrary intersections,
\item[(viii)] if $k\in\Z$ and $E\in\EE_G$, then $\{x\in G: kx\in E\}\in\EE_G$.
\end{itemize}
\end{lemma}

\begin{proof} 
(i) Replacing the multiplicative notation from Definition \ref{Markov's:definition}(i)  with the additive notation and using commutativity, one immediately gets
that $E\in \EE_G$ if and only if $E=\{x\in G: nx+b=0\}$ for a  suitable integer $n$ and an element $b\in G$. Replacing $b$ with $-b$, if necessary, we may assume that $n\in \N$. Choose $a\in E$. Then $na+b=0$, and so $E=\{x\in G: n(x-a)=0\}=\{x\in G: x-a\in G[n]\}=a+G[n]$.

(ii) Note that $G=G[0]\in\EE_G$ by (i). If $G=\{0\}$, then $\emptyset\in \EE_G$ according to  Notation \ref{Def:elem:alg:set}. If $G\not=\{0\}$,  choose $g\in G$ with $g\not=0$ and observe that $\emptyset=\{x\in G: 0x+g=0\}\in\EE_G$. 

(iii) and (iv) follow easily from (i). 

(v) Without loss of generality, we can assume that $n=2$ and  that both $E_1$ and $E_2$ are non-empty.  Applying (i), we conclude that $E_i = a_i + G[m_i]$  for some $a_i \in G$ and $m_i \in \N$ $(i=1,2)$. Let $m$ be the  least common multiple of  $m_1$ and $m_2$. Then $G[m] = G[m_1] + G[m_2]$, and (i) now yields
$$
E_1+E_2=a_1+G[m_1]+a_2+G[m_2]=(a_1+a_2) + G[m_1] + G[m_2]= (a_1+a_2) +G[m]\in  \EE_G.
$$

(vi) follows from (i) and Lemma \ref{intersection:of:two:elementary:algebraic:sets}(iv).

(vii) Assume that $m, n\in \N$, $a,b\in G$, $m\ge 1$, $a+G[n]$ is a proper subset of $b+G[m]$. Applying items (i) and (iii) of Lemma \ref{comparing:elementary:algebraic:sets}, we conclude that $G[n]=G[d]$ for a positive divisor $d$ of $m$. Since $m\ge 1$ has finitely many divisors, every descending chain of elementary algebraic sets stabilizes. Now the stability of  $\EE_G$  under arbitrary intersections follows from (vi).

(viii) The assertion is true if $G=\{0\}$, since  $\EE_G$  then coincides with the power set of $G$ according to Notation \ref{Def:elem:alg:set}.  Suppose now that $G\not=\{0\}$,
 and let $E'=\{x\in G: kx\in E\}$. If $E'= \emptyset$, then $E' \in \EE_G$ by  (ii). If $E'\ne \emptyset$, choose $x_0\in E'$. Then $kx_0\in E\not=\emptyset$, so $E=a+G[n]$ for some $a\in G$ and  $n\in \mathbb{N}$; see (i). Hence, $kx_0=a+t$ for some  $t\in G$ with $nt=0$. Let us check that $E'=x_0+G[kn]\in \EE_G$.  Assume $x\in E'$. Then $kx\in E$, and so $kx=a+s$ for some $s\in G[n]$. As $ kx_0=a+t$ with $t\in G[n]$, we conclude that $k(x-x_0)=s-t \in G[n]$, and so $ x-x_0\in G[kn]$. This proves that $x\in x_0+G[kn]$. In the opposite direction, if $ x\in x_0+G[kn]$, then $kn(x-x_0)=0$, so that $k(x-x_0)\in G[n]$,  and thus $kx-kx_0\in G[n]$. As $kx_0-a \in G[n]$, we conclude that $kx-a\in G[n]$ as well,  and so  $kx\in a+G[n]=E$. 
\end{proof}

\begin{lemma}\label{properties:of:algebraic:sets}
For every abelian group $G$, the family ${\mathfrak A}_G$ of all algebraic subsets of  $G$ has the following properties:
\begin{itemize}
\item[(i)] ${\mathfrak A}_G$ is closed under the operation $x\mapsto -x$,
\item[(ii)] ${\mathfrak A}_G$ is closed under taking translations,
\item[(iii)] $A_1, \ldots, A_n\in {\mathfrak A}_G$ implies $A_1+\ldots + A_n\in {\mathfrak A}_G$,
\item[(iv)] if $k\in\Z$ and $A\in{\mathfrak A}_G$, then  $\{x\in G: kx\in A\}\in{\mathfrak A}_G$.
\end{itemize}
\end{lemma}

\begin{proof} (i) follows from Lemma \ref{lemma:elementary:algebraic:sets}(iii);
(ii) follows from Lemma \ref{lemma:elementary:algebraic:sets}(iv);
(iii) follows from Lemma \ref{lemma:elementary:algebraic:sets}(v);
and (iv) follows from Lemma \ref{lemma:elementary:algebraic:sets}(viii).
\end{proof}

\begin{theorem} {\rm (\cite{Bryant})} \label{Zariski.topology.is.Noetherian} Let $G$ be an abelian group. Then:
\begin{itemize}
  \item[(i)] $(G, \Zar_G)$ is a Noetherian space, 
  \item[(ii)] ${\mathfrak A}_G$ coincides with the family of all $\Zar_G$-closed sets.
\end{itemize}
\end{theorem}

\begin{proof} In view of 
(ii), (vi), (vii) of Lemma \ref{lemma:elementary:algebraic:sets}, we can apply Fact  \ref{building:Noetherian:spaces} to conclude that ${\mathfrak A}_G=\fin{\EE_G}$ is the family of all closed sets of a unique Noetherian topology on $G$, and one can easily see that this topology coincides with $\Zar_G$.
\end{proof}

It follows from Theorem \ref{Zariski.topology.is.Noetherian}(i) that, for an abelian group $G$, the space $(G,\Zar_G)$ has all the basic properties of Noetherian spaces described in Facts \ref{Noetherian:facts} and \ref{decomposition:into:irreducible:components}. Clearly, infinite
groups $G$ with the discrete Zariski topology $\Zar_G$ (see Example \ref{discrete:Zariski}) are not Noetherian, and for them, all the properties listed in Facts \ref{Noetherian:facts} and  \ref{decomposition:into:irreducible:components} fail.

It follows from Theorem \ref{Zariski.topology.is.Noetherian}(ii) that, for a countable abelian group $G$, the Zariski topology $\Zar_G$ is countable as well. In the non-abelian case, one has a completely different situation. Indeed, let $G$ be a countable group $G$ such that $\Zar_G$ is discrete; see items (a) and (b) of Example \ref{discrete:Zariski}. 
Then $\Zar_G$ has cardinality $\cont$. In particular, every subset of $G$ is algebraic, so $G$ has $\cont$-many algebraic sets. Since $G$ has only countably many elementary algebraic sets,  not every algebraic set is a finite union of elementary algebraic sets (compare this with Theorem \ref{Zariski.topology.is.Noetherian}(ii)).

\begin{corollary}
\label{Zariski:topology:never:group}
For an abelian group $G$, the following conditions are equivalent:
\begin{itemize}
\item[(i)] $\Zar_G$ is Hausdorff,
\item[(ii)] $\Zar_G$ is sober,
\item[(iii)]  $G$ is finite,
\item[(iv)] $(G,\Zar_G)$ is a topological group.
\end{itemize}
\end{corollary}

\begin{proof} 
The proof of the implication
(i)~$\to$~(ii) is trivial. Since $(G,\Zar_G)$ is a Noetherian space by Theorem \ref{Zariski.topology.is.Noetherian}(i), the implication (ii)~$\to$~(iii) follows from Fact \ref{sober:fact}.
 
To prove the implication
(iii)~$\to$~(iv), note that the topology $\Zar_G$ is $T_1$, and since $G$ is finite, we conclude that $(G,\Zar_G)$ is discrete. In particular, $(G,\Zar_G)$ is a topological group.

Finally, the implication
(iv)~$\to$~(i) 
holds, since $(G,\Zar_G)$ is a $T_1$-space, and a $T_1$ topological group is Hausdorff.
\end{proof}

Corollary \ref{Zariski:topology:never:group} is typical for the abelian case
but fails for non-commutative groups.  Indeed, any infinite group $G$ with discrete $\Zar_G$ (see Example \ref{discrete:Zariski}) provides an example 
in which $(G,\Zar_G)$ is a Hausdorff (and thus  a sober) topological group.

\begin{corollary}
\label{continuity:of:some:operations:in:the:Zariski:topology}
Let $G$ be an abelian group. Then:
\begin{itemize}
   \item[(i)] the inverse operation $x\mapsto -x$ is $\Zar_G$-continuous,
   \item[(ii)] for every $a\in G$,
 the translation by $a$, $x\mapsto x+a$, is a homeomorphism of $(G,\Zar_G)$ onto itself,
   \item[(iii)] for every $k\in \Z$,
 the map $f_k:(G,\Zar_G)\to (G,\Zar_G)$, defined by $f_k(x)= kx$ for $x\in G$, is continuous.
\end{itemize}
\end{corollary}

\begin{proof} In view of Theorem \ref{Zariski.topology.is.Noetherian}(ii), 
(i)  follows from Lemma \ref{properties:of:algebraic:sets}(i);
(ii) follows from Lemma \ref{properties:of:algebraic:sets}(ii);
and (iii) follows from Lemma \ref{properties:of:algebraic:sets}(iv).
\end{proof}

Let $\mathcal{T}$ be a $T_1$-topology on an abelian group $G$.  If the ``inverse'' operation $x\mapsto -x$ is $\mathcal{T}$-continuous, and the addition operation $(x,y)\mapsto x+y$ is separately $\mathcal{T}$-continuous, then Kaplansky calls the pair $(G,\mathcal{T})$ a {\em $T_1$-group\/}; see \cite[Chap. IV, p. 27]{K}. The same pair is often called a {\em quasi-topological group\/}; see \cite{AT}. 

\begin{corollary}
\label{quasi-group}
If a group $G$ is abelian, then $(G, \Zar_G)$ is a $T_1$-group  (that is, a quasi-topological group).
\end{corollary}

According to Kaplansky, a {\em $Z$-group\/} is a $T_1$-group with a Noetherian topology, and a {\em $C$-group\/} is a $T_1$-group such that, for any fixed $a\in G$, the map $x\mapsto  x^{-1}ax$  is continuous; see \cite[Chap. IV, p. 28]{K}. According to Bryant and Yen \cite{BY}, a {\em $CZ$-group\/} is a $C$-group that is also a $Z$-group.
From Theorem \ref{Zariski.topology.is.Noetherian}(i) and Corollary \ref{quasi-group}, it follows that $(G, \Zar_G)$ is a $CZ$-group.

 The Zariski topology $\Zar_G$ is always $T_1$. The next fact  completely describes the abelian groups $G$ for which $\Zar_G$  is the coarsest $T_1$ topology, i.e., the cofinite topology. 

\begin{fact}\label{necessity:almost_torsion-free} 
{\rm (\cite[Theorem 5.1]{TY})}. For an abelian group $G$, the following conditions are equivalent:
\begin{itemize}
 \item[(i)] every proper algebraic subset of $G$ is finite; that is, $\Zar_G$ coincides with the  cofinite topology of $G$,
 \item[(ii)] every proper elementary algebraic subset of $G$ is finite,  
 \item[(iii)] either $G$ is almost torsion-free, or $G$ has exponent $p$ for some prime $p$.
\end{itemize}
\end{fact}

\begin{proof}  
The proofs of
(i)~$\leftrightarrow$~(ii) and (iii)~$\to$~(ii) are obvious. Let us prove (ii)~$\to$~(iii). Assume that (ii) holds, and $G$ is not almost torsion-free. Then there exists a prime $p$ such that $r_p(G)$ is  infinite. Then  the subgroup $G[p]$ of $G$ is infinite. Since $G[p]$ is an elementary algebraic subset of $G$, (ii) implies that $G=G[p]$. 
\end{proof}

As usual, for a subset $Y$ of a topological space $(X,\mathcal{T})$, we denote by  $\mathcal{T}\restriction_Y$ the subspace topology $\{Y\cap U:U\in\mathcal{T}\}$ generated by $\mathcal{T}$.

The next lemma shows that the Zariski topology behaves well under taking subgroups. This is a typical property in the abelian case (for counter-examples in the  non-abelian case, see \cite{DS_JGT}). Although this lemma can be deduced from \cite[Lemma 2.2(a), Lemma 3.7 and Corollary 5.7]{DS_JGT}, we prefer to give a direct, short and transparent  proof here for the reader's convenience. 

\begin{lemma}\label{lemma:hereditary:Zariski} For every subgroup $H$ of an abelian group $G$, one has $\Zar_G\restriction_H=\Zar_H$.
\end{lemma}

\begin{proof}  In view of Theorem \ref{Zariski.topology.is.Noetherian}(ii) and the formulae $\fin{\EE_G}=\mathfrak{A}_G$,  $\fin{\EE_H}=\mathfrak{A}_H$, it suffices to check that
$\EE_H=\{H\cap A: A\in {\EE_G}\}$. Note that 
\begin{equation}
\label{eq:A:B} h+H[n] = H\cap (h+G[n])\mbox{ for all } n\in\N \mbox{ and } h\in H.
\end{equation}

Assume that $\emptyset\not=B\in\EE_H$. Then $B=h+H[n]$ for some $h\in H$ and $n\in\N$; see Lemma \ref{lemma:elementary:algebraic:sets}(i).
Since $A=h+G[n]\in\EE_G$ by Lemma \ref{lemma:elementary:algebraic:sets}(i), and $B=H\cap A$ by (\ref{eq:A:B}), we conclude that  $\EE_H\subseteq\{H\cap A: A\in \EE_G\}$.

Let $A\in \EE_G$.  Assuming that $H \cap A\ne\emptyset$,  choose $h\in H\cap A$. From $\emptyset\not=A\in\EE_G$ and Lemma \ref{lemma:elementary:algebraic:sets}(i), it follows that $A=a+G[n]$ for some $a\in G$ and $n\in\N$. According to Lemma \ref{comparing:elementary:algebraic:sets}(ii), one has $A=a+G[n]=h+G[n]$, and so $H\cap A=h+H[n]\in \EE_H$ by (\ref{eq:A:B}) and Lemma \ref{lemma:elementary:algebraic:sets}(i). This  proves that $\{H\cap A: A\in \EE_G\}\subseteq \EE_H$.
\end{proof}

\begin{remark}
Lemma \ref{lemma:hereditary:Zariski} on inclusions $H\hookrightarrow G$  cannot be extended to arbitrary homomorphisms $f:H\to G$.  Indeed, define $H=\Q$,  $G=\Q/\Z$, and let $f:H\to G$ be the  canonical  homomorphism. We claim that $f$  is not continuous when both groups have Zariski topologies. To see this, note that  the subset $F=\{0\}$
of  $G$ is $\Zar_G$-closed, while $\Z=f^{-1}(F)$ is not $\Zar_H$-closed,
since $\Zar_H$ is the cofinite topology of  $H$ by Fact \ref{necessity:almost_torsion-free}.
Therefore, $\Zar_G$ is not a functorial topology in Charles sense; see \cite[\S 7]{Fuchs}. 
\end{remark}

Our next result establishes two fundamental properties of the Zariski closure.

\begin{theorem}
\label{translation:of:closures}
Let $G$ be an abelian group. Then:
\begin{itemize}
\item[(i)] $\CL_{\Zar_G}(A+B)=\CL_{\Zar_G}(A)+\CL_{\Zar_G}(B)$ whenever $A\subseteq G$ and $B\subseteq G$,
\item[(ii)] $\CL_{\Zar_G}(a+S)=a+\CL_{\Zar_G}(S)$ for each $a\in G$ and every subset $S$ of $G$.
\end{itemize}
\end{theorem}

\begin{proof}
(i) According to Lemma \ref{properties:of:algebraic:sets}(iii), the set $\CL_{\Zar_G}(A)+\CL_{\Zar_G}(B)$ is $\Zar_G$-closed and contains $A+B$, which implies that
\begin{equation}
\label{eq.2}
\CL_{\Zar_G}(A+B)\subseteq \CL_{\Zar_G}(A)+\CL_{\Zar_G}(B).
\end{equation}
Let us prove the inverse inclusion. All translations $x\mapsto x+a$ ($a\in G$) are $\Zar_G$-continuous by Corollary \ref{continuity:of:some:operations:in:the:Zariski:topology}(ii), and so 
\begin{equation}
\label{eq.3}
\CL_{\Zar_G}(C)+D\subseteq \CL_{\Zar_G}(C+D)\mbox{ for every pair of subsets } C,D \mbox{ of } G.
\end{equation}
We claim that
\begin{equation}
\label{equation:AB}
\CL_{\Zar_G}(A)+\CL_{\Zar_G}(B)\subseteq \CL_{\Zar_G}(A+\CL_{\Zar_G}(B))\subseteq \CL_{\Zar_G}(\CL_{\Zar_G}(A+B))=\CL_{\Zar_G}(A+B).
\end{equation}
Indeed, the first inclusion in (\ref{equation:AB}) is obtained by  applying
(\ref{eq.3}) to $C=A$ and $D=\CL_{\Zar_G}(B)$, and the second inclusion in (\ref{equation:AB}) is obtained by applying
(\ref{eq.3}) 
to $C=B$ and $D=A$.  Combining (\ref{eq.2}) and (\ref{equation:AB}), we get (i).

(ii) Since $(G,\Zar_G)$ is a $T_1$ space, 
(ii) is a particular case of 
(i) applied to $A=\{a\}$ and $B=S$. Note that
(ii) also  follows directly from Corollary \ref{continuity:of:some:operations:in:the:Zariski:topology}(ii).
\end{proof}

According to Theorem \ref{translation:of:closures}, the addition function $\mu: G \times G \to G$ defined by $\mu(x,y) = x + y$  
maps the $\Zar_G \times\Zar_G$-closure of a {\em rectangular\/} set $A \times B$ to the $\Zar_G$-closure of the image $\mu(A\times B)$. This fact 
is significant,  since this does not hold for {\em arbitrary\/} subsets of $G\times G$, as the following remark demonstrates.
\begin{remark}
{\em For an infinite group $G$, there always exists a set $X\subseteq G \times G$, such that $\mu$ 
maps the $\Zar_G \times \Zar_G$-closure of $X$ outside of the $\Zar_G$ closure of $\mu(X)$.\/} Indeed, assume that $\mu$ maps
the $\Zar_G \times \Zar_G$-closure of $X$ in the $\Zar_G$ closure of $\mu(X)$  for {\em every\/} subset $X$ of $G \times G$. Then the map $\mu: (G \times G, \Zar_G \times \Zar_G) \to (G, \Zar_G) $ must be continuous. That is, the addition $(x,y)\mapsto x+y$ becomes $\Zar_G$-continuous. Since the inverse operation $x\mapsto -x$ is always $\Zar_G$-continuous by Corollary \ref{continuity:of:some:operations:in:the:Zariski:topology}(i), it follows that $(G, \Zar_G)$ is a topological group. Thus, $G$ must be finite by 
Corollary \ref{Zariski:topology:never:group},  which is a contradiction.
\end{remark}

One can easily find an $X$ as in the above example in the case $G=\Z$, or any almost torsion-free infinite group. Indeed,  let $Y$ be a countably infinite subset of $G$.
Since $G$ carries a cofinite topology by Fact \ref{necessity:almost_torsion-free}, one can easily check that the set $X=\{(y,-y): y\in Y\}\subseteq G \times G$ is  $\Zar_G \times\Zar_G$-dense, and so  $G\times G=\CL_{\Zar_G\times \Zar_G} (X)$. On the other hand, $\mu(X)=\{0\}$, and
so $\CL_{\Zar_G}(\mu(X))=\{0\}$. This shows that $\mu(\CL_{\Zar_G\times \Zar_G} (X))\setminus \CL_{\Zar_G}(\mu(X))= \mu(G\times G)\setminus\{0\}=G\setminus\{0\}\not=\emptyset$.

\begin{definition}
\label{definition:of:Zariski:atom}
A countably infinite subset $X$ of an abelian group $G$ will be called a {\em \Zc{\Zar_G}\/}  provided that 
$\Zar_G\restriction_X=\{X\setminus F:$ is a finite subset of $X\}\bigcup \{\emptyset\}$ is the cofinite topology of $X$.
\end{definition}
 
We refer the reader to Fact \ref{FactZcurve} for other conditions equivalent to the condition from Definition \ref{definition:of:Zariski:atom}.

Obviously, an infinite subset of a \Zc{\Zar_G} is still a \Zc{\Zar_G}. 

The Zariski closures of \Zc{\Zar_G}s are precisely the {\em irreducible\/}  (and hence, elementary) algebraic sets; one of the  implications  is proved in item (i) of the next proposition, while the other one  is given in Corollaries \ref{closure:irreducible:sets} and \ref{closures:of:atoms}. 

\begin{proposition}
\label{general:properties:of:atoms}
Let $X$ be a \Zc{\Zar_G} of an abelian group $G$. Then:
\begin{itemize}
\item[(i)] $\CL_{\Zar_G} (X)$ is irreducible,
\item[(ii)] there exists $a\in G$ and $n\in\mathbb{N}$ such that $\CL_{\Zar_G} (X)=a+G[n]$,
\item[(iii)] if $F$ is a $\Zar_G$-closed subset of $G$, then either $F\cap X$ is finite or $X\subseteq F$,
\item[(iv)] any faithfully enumerated sequence $\{x_k:k\in\mathbb{N}\}$ of points of $X$ converges to every point $x\in \CL_{\Zar_G} (X)$.
\end{itemize}
\end{proposition}
\begin{proof}
(i) Clearly, $X$ is irreducible. Therefore, $\CL_{\Zar_G} (X)$ must be irreducible by Fact \ref{Fact:irreducible}(iii).

(ii) Note that $\CL_{\Zar_G} (X)\in {\mathfrak A}_G$ by Theorem \ref{Zariski.topology.is.Noetherian}(ii), so $\CL_{\Zar_G} (X)=\bigcup\mathcal{F}$ for some finite family  $\mathcal{F}\subseteq \EE_G\subseteq \Zar_G$. Since $\CL_{\Zar_G} (X)$ is irreducible by item (i), applying Fact \ref{fact:on:finite:unions},
 we conclude that $\CL_{\Zar_G} (X)=F$ for some $F\in\mathcal{F}$.  Now (ii)  follows from Lemma \ref{lemma:elementary:algebraic:sets}(i).

(iii) Let $F$ be a $\Zar_G$-closed subset of $G$. Then $F\cap X$ is a closed subset of $(X,\Zar_G\restriction_X)$. 
Since $X$ is a \Zc{\Zar_G}, from  Fact \ref{FactZcurve}(b) one concludes that either   $F\cap X$ is finite or $F\cap X=X$ (and hence, $X\subseteq F$).

(iv) Let $\{x_k:k\in\mathbb{N}\}\subseteq X$ be a faithfully indexed sequence. Fix an arbitrary point $x\in \CL_{\Zar_G} (X)$. Let $U$ be a $\Zar_G$-open subset of $G$ containing $x$. Then $U\cap X$ is a non-empty $\Zar_G\restriction_X$-open subset of $X$. Since $X$ is a \Zc{\Zar_G}, from  Fact \ref{FactZcurve}(b) one concludes that 
$U\cap X=X\setminus T$ for some finite subset $T$ of $X$. In particular, $X\setminus T\subseteq U$.  It follows that $U$ contains all but finitely many elements of the sequence $\{x_k:k\in\mathbb{N}\}$. Since $U$ was taken arbitrarily,  we conclude that the sequence $\{x_k:k\in\mathbb{N}\}$ converges to $x$.
\end{proof}

\section{Essential order and Zariski irreducible subsets}
\label{section:4}

For a subset  $X$ of an abelian group $G$ and a natural number $n$, we define $nX=\{nx:x\in X\}$. The group $G$ is {\em bounded\/} if  $nG=\{0\}$ for some $n\in \mathbb{N}\setminus\{0\}$, and $G$ is {\em unbounded\/} otherwise.

We say that $d\in\mathbb{N}$ is a {\em proper divisor of $n\in\mathbb{N}$\/} provided that $d\not\in\{0,n\}$ and $dm=n$ for some $m\in \mathbb{N}$. 
Note that, according to our definition, each $d\in\mathbb{N}\setminus \{0\}$ is a proper divisor of $0$.

\begin{definition} 
\label{order:of:a:subset}
Let $n\in\mathbb{N}$. An abelian group $G$ is said to be {\em of exponent\/} $n$ if $nG=\{0\}$ but $dG\ne \{0\}$ for every proper divisor $d$ of $n$.  In this case,
 we call $n$ the {\em exponent of $G$\/}. 
\end{definition}

In particular, $G[n]$ has exponent $n$  precisely when $G[n]\ne G[d]$ for every proper divisor $d$ of $n$.  Note that the only  abelian group of exponent $1$ is $\{0\}$, while a group $G$ is of exponent $0$ precisely when $G$ is unbounded (i.e., $nG\ne \{0\}$ for every positive integer $n$). 

\begin{lemma}
\label{exponent:of:G[n]:can:be:assumed:equal:to:n}
If $G$ is an abelian group and $m\in\N$, then there exists $n\in\N$ such that $G[m]=G[n]$ and $G[n]$ has exponent $n$.
\end{lemma}

\begin{proof}
If $G[m]$ is unbounded, then $m=0$ and $G[m]$ has exponent $0$, so $n=m$ works. If $G[m]$ is bounded, then $G[m]$ has exponent $n$ for some divisor $n$ of $m$. Clearly, $G[m]=G[n]$ in this case.
\end{proof}

\begin{definition} Let $G$ be an abelian group.
\label{definition:of:essential:order}
\begin{itemize}
\item[(i)]  If $G$ is bounded, then the {\em essential order\/} $eo(G)$  of  $G$ is the smallest positive integer $n$ such that $nG$ is finite.
\item[(ii)] If $G$ is unbounded, we define $eo(G)=0$. 
\end{itemize}
\end{definition}

The notion of the essential order of a bounded abelian  group $G$, as well as the notation $eo(G)$, are due to Givens and Kunen \cite{GK},  although the definition in \cite{GK} is different (but equivalent) to ours.

\begin{lemma}
\label{lemma:4.2}
If $G$ be a bounded abelian group, $n=eo(G)$ and $mG$ is finite for some $m\in\mathbb{N}\setminus\{0\}$, then $n$ divides $m$.
In particular, the essential order $eo(G)$ of a bounded group $G$ divides its exponent. 
\end{lemma}
\begin{proof}
Clearly, $m\ge n$ by Definition \ref{definition:of:essential:order}(i). Assume that $n$ does not divide $m$. Then there exists  unique integers $q\in\mathbb{N}$ and $r$ satisfying $m = qn + r$ and $0< r < n$. Since $r = m - qn$, the group $rG  \subseteq mG + qnG$ is finite,  which contradicts $n=eo(G)$.
\end{proof}

\begin{lemma}
\label{lemma:4.2a}
Let $G$ be an abelian group and $n=eo(G)$. Then:
\begin{itemize}
  \item[(i)] $G[n]$ has finite index in $G$,
  \item[(ii)]  $eo(G[n])=eo(G)=n$.
\end{itemize}
\end{lemma}

\begin{proof} If $n=0$, then $G[n]=G[0]=G$,  and both items (i) and (ii) trivially hold. So we now assume that  $n\ge 1$. Then $G$ is a bounded torsion group.

(i) By Definition 
\ref{definition:of:essential:order}(i), $nG\cong G/G[n]$ must finite. Hence,  $G[n]$  is a finite index subgroup of $G$.

(ii) Let $m=eo(G[n])$. Since $G[n]$ is bounded, $m\ge 1$. By Lemma \ref{lemma:4.2}, $m$ divides $n$,  and so $G[m]\subseteq G[n]$ and $m\le n$. From 
(i) we conclude that $G[m]$ has  a finite index in $G[n]$. Since $G[n]$ has  a finite index in $G$, it follows that $G[m]$ has  a
finite index in $G$. Hence $mG\cong G/G[m]$ is finite, which implies $n=eo(G)\le m$. This proves that $m=n$.
\end{proof}

The next theorem describes the connected component and the irreducible component of the identity of an abelian group in the Zariski topology.

\begin{theorem}
\label{essential:order:and:components}
Let $G$ be an infinite abelian group and $n=eo(G)$. Then:
\begin{itemize}
\item[(i)] $G[n]$ is a $\Zar_G$-clopen subgroup of $G$,
\item[(ii)] $G[n]$ is both the connected component and the irreducible component of the identity of the space $(G,\Zar_G)$.  
\end{itemize}
\end{theorem}

\begin{proof}
(i)  By Lemma \ref{lemma:4.2a}(i), $G[n]$ has  a finite index in $G$. Since each of the finitely many pairwise disjoint cosets of $G[n]$ are $\Zar_G$-closed, all of them must also be $\Zar_G$-open. Therefore, $G[n]$ is $\Zar_G$-clopen.
 
(ii)
The irreducible component $E$ of $0$ must be an elementary algebraic set,  and
so $E=a+G[m]$ for suitable $a\in G$ and $m\in\mathbb{N}$; see Lemma \ref{lemma:elementary:algebraic:sets}(i). Since $0\in E$, we conclude that $E=G[m]$. From Corollary \ref{continuity:of:some:operations:in:the:Zariski:topology}, we obtain that $a+G[m]$ is the irreducible component of $a$ for every $a\in G$. Since $(G,\Zar_G)$ is Noetherian by Theorem \ref{Zariski.topology.is.Noetherian}(i), the number of irreducible components of $(G,\Zar_G)$ must be finite 
according to 
\ref{decomposition:into:irreducible:components}.
Therefore, $G=G[m]+F$ for some finite subset $F$ of $G$.  Since the set $mG=mG[m]+mF=mF$ is finite, $n=eo(G)$ divides $m$ 
by Lemma \ref{lemma:4.2}. Therefore, $G[n]\subseteq G[m]$. By (i), $G[n]$ is a clopen subgroup of $G$. Since $G[m]$ is connected, we must have $G[m]=G[n]$.  
\end{proof}

\begin{corollary}
\label{when:G:is:connected:and/or:irreducible}
For an infinite abelian group $G$, the following conditions are equivalent:
\begin{itemize}
    \item[(i)] $(G,\Zar_G)$ is connected,
    \item[(ii)] $(G,\Zar_G)$ is irreducible,
    \item[(iii)] $G=G[n]$, where $n=eo(G)$.
\end{itemize}
\end{corollary}

\begin{corollary}
\label{connected:vs:irreducible} 
Let $E$ be an elementary algebraic set of an infinite abelian group $H$. Then:
\begin{itemize}
    \item[(i)] if $E$ is connected, then $E$ is irreducible,
   \item[(ii)] the irreducible components of $E$  are disjoint.
\end{itemize}
\end{corollary}

\begin{proof} Assume that $E\not=\emptyset$. According to Lemma \ref{lemma:elementary:algebraic:sets}(i), $E=a+H[m]$ for some $a\in H$ and $m\in \N$.  By Corollary \ref{continuity:of:some:operations:in:the:Zariski:topology}(ii),  all topological properties involved in 
(i) and (ii) are translation  invariant, so we can assume without loss of generality
that $E=H[m]$. Let $G=E$ and $n=eo(G)$. Lemma \ref{lemma:hereditary:Zariski} implies that $\Zar_G=\Zar_H\restriction_G$. Therefore, without loss of generality, we may also assume that $H=G$. Applying Lemma \ref{lemma:4.2a}(i) and Theorem \ref{essential:order:and:components}  to $G$, 
we conclude that $G[n]$ is an irreducible $\Zar_G$-clopen subgroup of $G$ of finite index, and $G[n]$ coincides with the connected component of the identity of $(G,\Zar_G)$. 

(i) Since $G=E$ is connected, we deduce that $E= G[n]$,  and so $E$ is irreducible. 

(ii) Let $g_1+G[n],\dots, g_k+G[n]$ be pairwise disjoint cosets of $G[n]$ such that $G=\bigcup_{1\le i\le k} g_i+G[n]$. Since $G[n]$ is irreducible, each $g_i+G[n]$ is also irreducible by Corollary \ref{continuity:of:some:operations:in:the:Zariski:topology}(ii).
It now follows  that $g_1+G[n],\dots, g_k+G[n]$ are precisely the irreducible components of $G=E$.
\end{proof}

\begin{example}
\label{non-disjoint:components}
Let $G=\Z(6)^{(\omega)}$.  Applying Theorem \ref{essential:order:and:components} to the  groups $E_1=G[2]$ and $E_2=G[3]$, we conclude that $E_1$ and $E_2$  are irreducible (and hence, connected) $\Zar_G$-closed  subsets of $G$. Since $0\in E_1\cap E_2\ne \emptyset$,  the algebraic set $F=E_1\cup E_2$ is a connected non-irreducible set in $G$. Furthermore, $E_1$ and $E_2$ are irreducible components of $F$. This shows that both items of Corollary \ref{connected:vs:irreducible} fail for (non-elementary) algebraic sets.
\end{example}

\begin{proposition}
\label{essential:order=irreducible+order}
For an abelian group $G$ and an integer $n\in\mathbb{N}$ the following conditions are equivalent:
\begin{itemize}
\item[(i)] $eo(G[n])=n$,
\item[(ii)] $G[n]$ is irreducible and has exponent $n$.
\end{itemize}
\end{proposition}
\begin{proof} 
Let $H=G[n]$.  By Lemma \ref{lemma:hereditary:Zariski}, $H$ is an irreducible subset of $(G,\Zar_G)$ if and only if $(H,\Zar_H)$  is irreducible.  Hence, we will argue with the group $H=H[n]$ instead of $G$. 

(i)~$\to$~(ii) As $eo(H)=n$,  Theorem \ref{essential:order:and:components}  implies that $(H,\Zar_H)$ is irreducible. Since $eo(H)=n$ and $nH= 0$, it follows that $H$ has exponent $n$. 

(ii)~$\to$~(i) If $n=0$, then $H$  is unbounded, as a group of exponent $0$. Hence, $eo(H)=0$  by Definition   \ref{definition:of:essential:order}(ii).  If $n=1$, then $H = \{0\}$, and so $eo(H)=1$. Assume now that  $n> 1$.  Suppose $m\in\mathbb{N}$, $1\le m<n$ and $mH$ is finite. 
As $mH\cong H/H[m]$ is finite, the closed subgroup $H[m]$ of $H$ has finite index, so it is a clopen subgroup of $H$. As $H$ is irreducible, this yields $H[m]= H$. Thus,
$mH=\{0\}$. Since $m<n$, this contradicts the fact that  the group $H$ has exponent $n$.  This contradiction proves that $mH$ is infinite for every $m\in\mathbb{N}$ satisfying $1\le m<n$. Since $nH=\{0\}$,  we get $eo(H)=n$.
\end{proof}

Let us give an example illustrating the usefulness of  Proposition \ref{essential:order=irreducible+order}.

\begin{example}
\label{divisible:example:of:irreducible:subgroups}
{\em Let $G$ be a divisible abelian group. If $p$ is a prime number, $n\in\mathbb{N}\setminus\{0\}$ and $G[p^n]$ is infinite, then $G[p^n]$ is an irreducible subgroup of $G$.\/} Indeed, let $B=\{x_i: i\in I\}$ be a minimal set of generators of the subgroup  $G[p]$; that is, $B$ is a base of $G[p]$ considered as a linear space over the field $\Z / p\Z$. Since $G$ is divisible, for every $i\in I$, there exists  $y_i\in G$ such that $p^{n-1}y_i = x_i$. Then $y_i\in  G[p^n]$ for every $i\in I$, the subgroup of $G$ generated by $\{y_i: i\in I\}$ coincides with $G[p^n]$ and $G[p^n] \cong \bigoplus_I \Z(p^n)$. Since $G[p^n]$ is infinite, $I$ must be infinite as well. This yields $eo(G[p^n]) = p^n$, and the implication (i)~$\to$~(ii)  from Proposition \ref{essential:order=irreducible+order}  shows that $G[p^n]$ is irreducible.
\end{example}

According to Pr\" ufer's theorem \cite[Theorem 17.2]{Fuchs}, every bounded abelian group $G$ is a direct sum of  cyclic groups, so  that
$$
G=\bigoplus_{p\in \Prm}\bigoplus_{s\in \mathbb{N}}\Z({p^s})^{(\kappa_{p,s})},
$$ 
where only {\em finitely many\/} of the cardinals $\kappa_{p,s}$ are nonzero; these cardinals are known as {\em  Ulm-Kaplansky invariants\/} of $G$.  For every $p\in\Prm$, the $\kappa_{p,s}>0$ with maximal $s$  is referred to as  the {\em leading Ulm-Kaplansky invariant of $G$ relative to $p$\/}. 

If $G$ is a bounded torsion abelian group and $n=eo(G)$, then it easily follows from Pr\" ufer's theorem that  there exists a decomposition $G=G_1\oplus F$, where $F$ is finite, $eo(G_1)=n$, and $G_1$ has exponent $n$. 

We will now  characterize the condition $eo(G[n])=n$ (for $n>1$) that appears in Proposition \ref{essential:order=irreducible+order} in terms of the Ulm-Kaplansky invariants of $G$. 

\begin{proposition}
\label{U-K:invar} 
Given an integer $n>1$ and an infinite abelian group $G$, the following conditions are equivalent: 
\begin{itemize}
    \item[(i)] $eo(G[n])=n$,
    \item[(ii)] all leading Ulm-Kaplansky invariants of $G[n]$ are infinite,
    \item[(iii)] there exists a monomorphism $\Z(n)^{(\omega)}\hookrightarrow G[n]$.
\end{itemize}
\end{proposition}
\begin{proof} 
(i)~$\to$~(ii) If $eo(G[n])=n$ and $p$ is a prime dividing $n$, then the leading Ulm-Kaplansky  invariant relative to $p$ is  determined by the $p$-rank of the subgroup $(n/p)G[n]$ of  exponent $p$. Now, by our hypothesis, this group is infinite, and hence,  the leading Ulm-Kaplansky  invariant relative to $p$ is infinite as well. 

(ii)~$\to$~(iii) 
Let $P$ be the finite set of primes $p$ dividing $n$. For every $p\in P$,  let $p^{k_p}$ be the  highest power of $p$  dividing $n$. By (ii), the leading Ulm-Kaplansky  invariant $\kappa_{p,k_p}$ of $G[n]$ is infinite. Hence,  the $p$-primary component of $G[n]$ has a direct summand of the form $\Z({p^{k_p}})^{(\kappa_{p,k_p})}$. Therefore,  there exists a  monomorphism $\Z(p^{k_p})^{(\omega)}\hookrightarrow G[n]$.  Since $\bigoplus_{p\in P}\Z(p^{k_p})^{(\omega)}\cong \Z(n)^{(\omega)}$, this implies (iii).

(iii)~$\to$~(i) Suppose that $m\in\mathbb{N}$ and $1\le m<n$.  Then the group $mG[n]$ contains the infinite subgroup $m\left(\Z(n)^{(\omega)}\right)\cong \Z(n/d)^{(\omega)}$, where $d=(m,n)$ is the greatest common divisor of $m$ and $n$.  Since $nG[n]=\{0\}$, this proves that $eo(G[n])=n$.
\end{proof}

\section{\Round{n} sets as building blocks for the Zariski topology}
\label{section:5}

\begin{definition}\label{def:of:almost:n:torsion:sets}
\begin{itemize}
\item[(i)]
For a given  $n\in \mathbb{N}$, we say that  a countably infinite subset $S$ of an abelian group $G$ is {\em \round{n} in $G$\/} if 
$S\subseteq G[n]$ and  the set $\{x\in S: dx=g\}$ is finite for  each $g\in G$ and every proper divisor $d$ of $n$; see \cite{DS}.
\item[(ii)] For $n\in \mathbb{N}$, let $\TT_n(G)$ denote the family of all  \round{n} sets in $G$.
\item[(iii)] Define  $\TT(G)=\bigcup\{\TT_n(G):n\in\N\}$.
\end{itemize}
\end{definition}

In order to clarify Definition \ref{def:of:almost:n:torsion:sets} and to facilitate future references, we collect basic properties of  \round{n} sets in our next remark.

\begin{remark}\label{Background} 
Let $G$ be an abelian group. 
\begin{itemize}
\item[(i)] $\TT_1(G)=\emptyset$; that is, there are no \round{0} sets.
\item[(ii)] $ \TT_n(G)\cap \TT_m(G)=\emptyset$ for distinct $m, n\in\N$.
\item[(iii)] Each family $\TT_n(G)$ is closed under taking infinite subsets, and so $\TT(G)$ has the same property.
\item[(iv)] If $H$ is a subgroup of $G$, then $\TT_n(H)= \{S\in\TT_n(G): S\subseteq H\}$ for every $n\in\N$; see \cite[Lemma 4.4]{DS}. In particular,  whether a set $S$ is \round{n} in $G$ depends only on the subgroup of $G$ generated by $S$. 
\end{itemize}
\end{remark}

\begin{notation} For an abelian group $G$ and  $S\in \TT(G)$, we use $\n_S$ to denote the unique integer $n\in\N$  such that $S\in \TT_n(G)$. 
(The uniqueness of such $n$ follows from Remark \ref{Background}(ii).)
\end{notation}

 The notion of \round{n} set was introduced first in  \cite[Definition 3.3]{DT} under a different name and has been even split into two cases;
see \cite[Remark 4.2]{DS} for an extended comparison between this terminology and the one proposed in  \cite{DT}.

The  \round{n} sets were  used in \cite{DS,DT} to build countably compact group topologies on abelian groups,  while they were used in \cite{TY} to construct independent group topologies on abelian groups. But only in the context of  the  Zariski topology  can one fully realize the true power of these remarkable sets due to their close relation to \Zc{\Zar_G}s; see Theorem \ref{corollary:about:atoms:being:translates:of:round:sets}. In fact, this relation  permits us to describe the Zariski topology of the abelian groups in full detail; see Sections 6 and 8. 

\begin{definition}
\label{definition:M(X)}
For a subset $X$ of an abelian group $G$, we define
\begin{equation}
\label{formula:for:M(X)}
M(X)=\{n\in\N\setminus\{0\}: X\cap (a+G[n]) \mbox{ is infinite for some } a\in G\}.
\end{equation}
and
\begin{equation}
\label{eq:for:mathfrak:n}
\mathfrak{m}(X)=
\left\{ \begin{array}{ll}
\min M(X), & \mbox{ if } M(X)\not=\emptyset  \\
0, & \mbox{ if } M(X)=\emptyset.
\end{array} \right.
\end{equation}
\end{definition}

We shall see in Corollary \ref{extension:of:function:n:to:all:subsets} below that the function
$X\mapsto \mathfrak{m}(X)$ is an extension of the function $S\mapsto \n_S$ from the family 
$\TT(G)$ to the family of {\em all\/} subsets of $G$. 

From Definition \ref{definition:M(X)},
 it immediately follows that $\mathfrak{m}(X)=\mathfrak{m}(g+X)$ for every $g\in G$ and each $X\subseteq G$. Our next lemma shows that in 
(\ref{formula:for:M(X)}),
 it suffices to consider only those $a\in G$ that are  elements of $X$.

\begin{lemma}
For every subset $X$ of an abelian group $G$, one has
\begin{equation}
\label{formula:for:M(X):new}
M(X)=\{n\in\N\setminus\{0\}: X\cap (x+G[n]) \mbox{ is infinite for some } x\in X\}.
\end{equation}
\end{lemma}

\begin{proof}
Assume that $X\cap (a+ G[n])$ is infinite for some $a\in G$ and $n\in\N$. Choose $x_0\in X\cap (a+ G[n])$. Since $x_0+ G[n]=a+G[n]$, the intersection $X\cap (x_0+ G[n])=X\cap (a+ G[n])$ must be infinite as well. This proves the non-trivial inclusion in (\ref{formula:for:M(X):new}). 
\end{proof}

The following lemma provides two reformulations of the notion of an \round{n} set.
\begin{lemma}
\label{equivalent:condition:for:round:sets}
Let $G$ be an abelian group and $n\in\mathbb{N}$. For a countably infinite set $S\subseteq G[n]$, the following conditions are equivalent:
\begin{itemize}
\item[(i)] $S$ is \round{n},
\item[(ii)] $S\cap (a+G[d])$ is finite whenever  $a\in G$  and $d$ is a proper divisor of $n$,
\item[(iii)] $n=\mathfrak{m}(S)$.
\end{itemize}
\end{lemma}
\begin{proof} The equivalence of  (i) and (ii) is  proven in \cite[Remark 3.4]{DT}. Let us  now  prove that (ii) and (iii) are also equivalent. We consider two cases.

{\sl Case 1\/}. $n=0$. Since every integer $n\in\mathbb{N}\setminus\{0\}$ is a proper divisor of $0$, condition (ii) is equivalent to $M(S)=\emptyset$ by (\ref{formula:for:M(X)}),
and the latter condition is equivalent to $\mathfrak{m}(S)=0$ by (\ref{eq:for:mathfrak:n}).

{\sl Case 2\/}. $n\in\mathbb{N}\setminus \{0\}$. 
Since $S\cap G[n]=S$ is infinite, $n\in M(S)\not=\emptyset$  by (\ref{formula:for:M(X)}).

Let us prove that
(ii)~$\to$~(iii). 
Let $m=\mathfrak{m}(S)$. Since $M(S)\not=\emptyset$, $m\ge 1$ by (\ref{eq:for:mathfrak:n}). Since $m\in M(S)$,  (\ref{formula:for:M(X)}) implies that the set $S'=S\cap (a+G[m])$ 
must be infinite for some $a\in G$. Since $S'\subseteq S\subseteq G[n]$, we have $S'\subseteq G[n]\cap (a+G[m])=z_0+G[d]$, 
where $z_0\in G$ and $d=(n,m)$; see Lemma \ref{intersection:of:two:elementary:algebraic:sets}(iv). 
Since $m\ge 1$ and $n\ge 1$, we have $d\ge 1$. Since   $S\cap (z_0+G[d])$ contains the infinite set $S'$, $d\in M(S)$ by (\ref{formula:for:M(X)}). Since $m=\min M(S)$, 
we conclude that $m\le d$. Since $d$ divides $m$, we must have  $d=m$. Since $d=m\in M(S)$, 
from (ii) it follows that $d$ cannot be a proper divisor of $n$. Hence, $n=d=m=\mathfrak{m}(S)$. 
This establishes (iii).

The reverse implication (iii)~$\to$~(ii) is obvious.
\end{proof}

\begin{corollary}
\label{extension:of:function:n:to:all:subsets}
If $G$ is an abelian group and $S\in\TT(G)$, then $\mathfrak{m}(S)=\n_S$.
\end{corollary}

\begin{corollary}\label{0-round:vs:n(-)=0}
A countably infinite subset $X$ of an abelian group $G$ is  \round{0} if and only if $\mathfrak{m}(X)=0$. 
\end{corollary}
\begin{proof} 
Since $X\subseteq G=G[0]$, the conclusion follows from the equivalence (iii)~$\leftrightarrow$~(i) of  Lemma \ref{equivalent:condition:for:round:sets}. \end{proof}

\begin{lemma}
\label{lemma:4.6}
Suppose that $m,n\in\mathbb{N}$, $S$ is a \round{n} subset of an abelian group $G$, $a\in G$ and $S\cap (a+G[m])$ is infinite. Then $G[n]\subseteq a+G[m]$.
\end{lemma}
\begin{proof}
By Lemma \ref{intersection:of:two:elementary:algebraic:sets}(iv), $S\subseteq G[n]\cap (a+G[m])=z_0+G[d]$ for some $z_0\in S$, where  $d=(n,m)$ is the greatest common divisor of $n$ and $m$. Since $S$ is infinite, from the implication (i)~$\to$~(ii) of Lemma \ref{equivalent:condition:for:round:sets},
 it follows that $d$ cannot be a proper divisor of $n$. Hence, $d=n$. Since $z_0\in S\subseteq G[n]$, we have  $z_0+G[d]=z_0+G[n]=G[n]$. We have 
thus proven that $G[n]\cap (a+G[m])=G[n]$, which yields $G[n]\subseteq a+G[m]$.
\end{proof}

\begin{proposition}
\label{round:sets:as:dense:atoms}
Given $n\in\mathbb{N}$ and a countably infinite subset $S$ of an abelian group $G$, the following conditions are equivalent:
\begin{itemize}
\item[(i)] $S$ is \round{n},
\item[(ii)] $G[n]$ has exponent $n$, $\CL_{\Zar_G} (S)=G[n]$ and  $S$ is a \Zc{\Zar_G}.
\end{itemize}
\end{proposition}
\begin{proof}
(i)~$\to$~(ii) 
Suppose  that   $d(G[n]) =0$ (i.e., $G[n]=G[d]$) for some proper divisor $d$ of $n$.    Then  $\{x\in S: dx=0\}=\{x\in S: nx=0\}=S$, because  $S \subseteq G[n]$. Since $S$ is infinite, $S$ is not \round{n}, which contradicts (i). This proves that $G[n]$ has exponent $n$.

Let $F$ be a $\Zar_G$-closed set. According to Theorem \ref{Zariski.topology.is.Noetherian}(ii),  there exist  $k\in\mathbb{N}\setminus\{0\}$, $a_1,...,a_k\in G$ and positive integers $n_1,...,n_k$ such that  $F=(a_1+G[n_1])\cup  (a_2+G[n_2])\cup \ldots \cup (a_k+G[n_k])$. If $S\cap F$ is infinite, then $S\cap (a_i+G[n_i])$ must be infinite for some $i\le k$.  Now Lemma \ref{lemma:4.6}  implies that
$S\subseteq G[n]\subseteq a_i+G[n_i]\subseteq F$. This shows that $S$ is a \Zc{\Zar_G} satisfying $G[n]\subseteq \CL_{\Zar_G} (S)$. Since $S\subseteq G[n]$ and $G[n]$ is $\Zar_G$-closed, we also have  the reverse inclusion,
namely,  $\CL_{\Zar_G} (S)\subseteq G[n]$.

(ii)~$\to$~(i)
It suffices to prove that $S$ satisfies the condition (ii) of Lemma \ref{equivalent:condition:for:round:sets}. Let $a\in G$  and $d$ be a proper divisor of $n$. Since $G[n]$ has exponent $n$, the set $E=a+G[d]$ is either disjoint from $G[n]$ or a proper closed subset of $G[n]$.  In the first case,
 $S \cap E= \emptyset$.  In the second case, since
$E$ is a proper closed subset of $G[n]$ and $\CL_{\Zar_G} (S)=G[n]$, $S\cap E$ must be a proper closed subset of $S$. Since $S$ is a \Zc{\Zar_G},  it follows that  $S \cap E$ must be finite. 
\end{proof}

\begin{lemma}
\label{set-theoretic:lemma}
Suppose that $Y$ is a subset of an abelian group $G$ and $\mathcal{H}\not=\emptyset$ is a countable family of subgroups of $G$ such that $Y\setminus \bigcup_{i\le k} (a_i+H_i)\not=\emptyset$ whenever $k\in\mathbb{N}$, $a_0,\dots,a_k\in G$ and $H_0,\dots,H_k\in\mathcal{H}$. Then there exists an infinite set $S\subseteq Y$ such that $S\cap (a + H)$ is finite whenever $a\in G$ and $H\in\mathcal{H}$. 
\end{lemma}

\begin{proof} Let $\mathcal{H}=\{H_n:n\in\mathbb{N}\}$ be an enumeration of $\mathcal{H}$. For a finite set $F\subseteq G$ and $k\in\mathbb{N}$,
 define $E_{F,k}=\bigcup_{i\le k}F+H_i$,
 and note that $Y\setminus E_{F,k}\not=\emptyset$ by the assumption of our lemma. Therefore, by induction on $k$, we can choose a sequence $\{y_k:k\in\mathbb{N}\}$ such that $y_{k+1}\in Y\setminus E_{F_k,k}$ for every $k\in\mathbb{N}$, where  $F_k=\{y_0,\dots,y_k\}$. Indeed, $Y\not=\emptyset$ allows one to choose $y_0\in Y$. Assuming that $y_0,\dots,y_k$ have already been chosen, select $y_{k+1}\in Y\setminus E_{F_k,k}\not=\emptyset$.

We claim that $S=\{y_k:k\in\mathbb{N}\}$ is as required. Clearly, $S\subseteq Y$. Since $\mathcal{H}\not=\emptyset$, from our choice of $y_{k+1}$,
 it follows that $y_{k+1}\not\in F_k + H_0=\{y_0,\dots,y_k\}+ H_0$. Since $0\in H_0$, we conclude that $y_{k+1}\not\in\{y_0,\dots,y_k\}$. Therefore, $S$ is infinite. Assume now that $a\in G$ and $H\in\mathcal{H}$. Then there exists $n\in\mathbb{N}$ such that $H=H_n$. We  must show that the set $S\cap (a+H)=S\cap (a+H_n)$ is finite. If this set is empty, then the proof is complete. Suppose now that $y_m\in a+H_n$ for some $m\in\mathbb{N}$. Define $k=\max(m,n)$. Let $j>k$ be an arbitrary integer. Suppose that $y_j\in a+H_n$. Then  $y_j-y_m\in H_n$, and so 
$$
y_j\in y_m+H_n\subseteq F_k+H_n\subseteq F_{j-1}+H_n\subseteq E_{F_{j-1}, j-1},
$$
which contradicts the choice of $y_j$. This proves that $S\cap (a+H_n)\subseteq \{y_0,\dots,y_k\}$.
\end{proof}

Our next proposition is the main result of this section. It provides an important characterization of  sets containing an \round{n} set that will be used frequently in subsequent sections.

\begin{proposition}
\label{finding:n-round:sets}
Suppose that $n\in\mathbb{N}$ and $G$ is an abelian group. Then for every infinite set $Y\subseteq G[n]$, the following conditions are equivalent:
\begin{itemize}
     \item[(i)] $Y$ contains an \round{n} set,
     \item[(ii)]  $G[n]$ is an irreducible subgroup of $G$ of exponent $n$, and $\CL_{\Zar_G}(Y)= G[n]$,
     \item[(iii)] $eo(G[n])=n$ and $\CL_{\Zar_G}(Y)= G[n]$.
\end{itemize}
\end{proposition}
\begin{proof} 
Since $G[1]=\{0\}$ contains no infinite subsets $Y$, our proposition trivially holds for $n=1$. Therefore, for the rest of the proof, we  assume that  $n\in\mathbb{N}\setminus\{1\}$. 

(i)~$\to$~(ii)  Let $S$ be an \round{n} subset of $Y$. Applying the implication (i)~$\to$~(ii) of Proposition \ref{round:sets:as:dense:atoms}, we conclude that $S$ is \Zc{\Zar_G},
$G[n]$ has exponent $n$,
 and $G[n]=\CL_{\Zar_G}(S)\subseteq \CL_{\Zar_G}(Y)\subseteq  G[n]$.
In particular, $\CL_{\Zar_G}(Y)=  G[n]$. Since $S$ is a \Zc{\Zar_G}, $G[n]=\CL_{\Zar_G}(S)$ is irreducible by Proposition \ref{general:properties:of:atoms}(i).

(ii)~$\to$~(i)  
Define $\mathcal{H}=\{G[d]:d\in D\}$, where  $D$ is the set of all proper divisors of $n$.  We claim that $G$, $Y$ and $\mathcal{H}$ satisfy the assumptions of Lemma \ref{set-theoretic:lemma}.  Indeed,  observe that $\{0\}=G[1]\in\mathcal{H}\not=\emptyset$, since
$1$ is a proper divisor of every  $n\in\mathbb{N}\setminus\{1\}$. Furthermore, suppose that $k\in\mathbb{N}$, $a_0,\dots,a_k\in G$ and $H_0,\dots,H_k\in\mathcal{H}$. Assume that $Y\subseteq E=\bigcup_{i\le k} (a_i+H_i)$. Since $E$ is $\Zar_G$-closed,  $G[n]=\CL_{\Zar_G}(Y)\subseteq E$.  Since $G[n]$ is irreducible, $G[n]\subseteq a_i+H_i$ for some $i\le k$ by Fact \ref {fact:on:finite:unions}.  As $H_i=G[d_i]$ for some $d_i\in D$,  Lemma   \ref{comparing:elementary:algebraic:sets}(i)  yields $G[n]\subseteq G[d_i]$. From $d_i\in D$, it follows that  $d_i$ is a proper divisor of $n$,  and  so  $G[d_i]\subseteq G[n]$. Thus, $G[d_i]=G[n]$, which contradicts our assumption that $G[n]$ has exponent $n$. This  shows that $Y\setminus \bigcup_{i\le k} (a_i+H_i)\not=\emptyset$.

Let $S$ be  a set from the conclusion of Lemma  \ref{set-theoretic:lemma}. Then $S$ is \round{n} by implication (ii)~$\to$~(i) of Lemma \ref{equivalent:condition:for:round:sets}.

(ii)~$\leftrightarrow$~(iii) 
This follows from  Proposition \ref{essential:order=irreducible+order}.
\end{proof}

\begin{corollary}
\label{description:of:when:there:are:n-round:sets}
For a given integer $n\in\mathbb{N}$, an abelian group $G$ contains an \round{n} subset if and only if $eo(G[n])=n$. 
\end{corollary}

\begin{proof}
Apply Proposition \ref{finding:n-round:sets} to $Y=G[n]$. 
\end{proof}

Our next proposition provides a  simple characterization of  the sets that contain an \round{0} set. It will be used in Theorem \ref{Zariski:dense:sets:in:unbounded:groups}  to describe the Zariski dense subsets of an unbounded abelian group $G$.

\begin{proposition}
\label{sets:that:contain:zero-round:sets}
For a subset $Y$ of an abelian group $G$, the following conditions are equivalent:
\begin{itemize}
\item[(i)] $Y$ contains an \round{0} set,
\item[(ii)] $mY$ is infinite for every  integer $m\in\mathbb{N}\setminus\{0\}$,
\item[(iii)] $Y\setminus (G[m]+F)\not=\emptyset$ for every integer $m\in\mathbb{N}\setminus\{0\}$ and each finite subset $F$ of $G$.
\end{itemize} 
\end{proposition}
\begin{proof}
(i)~$\to$~(ii)
Let $S$ be an \round{0} subset of $Y$. Assume that $mS$ is finite. Then we can find  $s_0\in S$ such that $ms=ms_0$ for infinitely many $s\in S$. Since $m$ is a proper divisor of $0$, this contradicts Definition \ref{def:of:almost:n:torsion:sets}. This shows that $mS$ (and then $mY$ as well) must be infinite.

(ii)~$\to$~(iii) Assume $Y\subseteq G[m]+F$ for some integer $m\in\mathbb{N}\setminus\{0\}$ and some finite subset $F$ of $G$. Then $mY \subseteq mF$ is finite, 
which contradicts 
(ii).

(iii)~$\to$~(i)  We claim that $G$, $Y$ and 
$\mathcal{H}=\{G[n]:n\in \mathbb{N}\setminus\{0\}\}$ satisfy the assumptions of Lemma \ref{set-theoretic:lemma}.  Suppose that $k\in\mathbb{N}$, $a_0,\dots,a_k\in G$ and $H_0,\dots,H_k\in\mathcal{H}$. Then there exist $m_0,\dots,m_k\in\mathbb{N}\setminus\{0\}$ such that $H_i=G[m_i]$ for $i\le k$.  Let $m=m_0 m_1\dots m_k$. Then 
$$
\bigcup_{i\le k}(a_i+H_i)=\bigcup_{i\le k}(a_i+G[m_i])\subseteq \bigcup_{i\le k}(a_i+G[m])=\{a_0,\dots,a_k\}+G[m],
$$
and 
(iii)
implies that $Y\setminus \bigcup_{i\le k}(a_i+H_i)\not=\emptyset$.

Let $S$ be a set from the conclusion of Lemma  \ref{set-theoretic:lemma}.  Since (iii) obviously implies that $G$ is unbounded, $G$ has exponent $0$.
Since $\mathbb{N}\setminus\{0\}$ is the set of all proper divisors of $0$ and $S$ satisfies 
Lemma  \ref{set-theoretic:lemma}, $S$ is \round{0} by the implication (ii)~$\to$~(i) of Lemma \ref{equivalent:condition:for:round:sets}.
\end{proof}

\begin{corollary}
\label{groups:that:conatin:0-round:sets}
An abelian group $G$ contains an \round{0} subset if and only if $G$ is unbounded.
\end{corollary}

\begin{proof}
This follows from Proposition \ref{sets:that:contain:zero-round:sets}  when $Y=G$. Alternatively, 
this also follows from Corollary \ref{description:of:when:there:are:n-round:sets} 
when $n=0$ and Definition \ref{definition:of:essential:order}(ii).
\end{proof}

There is another, direct proof of  this corollary. Clearly, bounded groups have no \round{0} subsets.  Now assume that $G$ is   unbounded, and choose for every $n>0$ an element $s_n\in G$ such that $n!s_n\ne 0$. If $G$ is non-torsion, then any  infinite cyclic subgroup of $G$ is an \round{0} subset, so we shall assume that $G$ is torsion.
In this case, $S=\{s_n:n\in \N\}$ is an  \round{0} subset of $G$.  Indeed,  consider a non-zero $d\in \N$ and $h\in G$. It suffices to prove that the equation $ds_n=h$  holds only for finitely many $n\in\N$. Since $G$ is torsion,  $ n_0h=0$  for some $n_0>0$.  Suppose that $n\in\N$ and $n\geq n_0+d$. Then  $n_0 d | n!$, and  $n!s_n\ne 0$  
yields $n_0 ds_n\ne 0$.   Since $n_0h=0$, it follows that $ds_n \ne h$.

\begin{proposition}
\label{almost:torsion-free:groups:by:means:of:n-round}
An infinite abelian group $G$ is almost torsion-free if and only if $G$ has no \round{n} sets for any integer $n>0$.
\end{proposition}

\begin{proof} 
Assume that $G$ is almost torsion-free. For every prime $p$, the $p$-rank $r_p(G)$ of $G$ is defined as the dimension of the subgroup $G[p]$ over the field $\Z/p\Z$.  Hence,
 the subgroup $G[p^m]$ is finite for every $m\in \N$. Now fix a positive $n\in \N$. The subgroup $G[n]$ of $G$ is a sum of subgroups of the form $G[p^m]$, 
where the prime powers $p^m$ run over all divisors of the form $p^m$ in the prime factorization of $n$; hence, $G[n]$ is finite. This proves that $G$ has no \round{n} sets. 

On the other hand, if $p$ is an arbitrary prime and $G$ has no \round{p} sets $p$, then $r_p(G)<\infty$.  Therefore, $G$ is almost torsion-free. 
\end{proof}

\section{Irreducible sets and \Zc{\Zar_G}s}
\label{section:6}

\begin{lemma}\label{L_translates:of:n-round}
Assume that $n\in \N$, $S$ is an \round{n} subset of an abelian group $G$ and $a\in G$.  Then $a+S$ is \round{n} if and only if $a\in G[n]$.
\end{lemma}

\begin{proof} Assume
that  $a+S$ is \round{n}. Then $a+S \subseteq G[n]$. Since $S\subseteq G[n]$, we conclude that $a\in G[n]$. 

Assume now that $a\in G[n]$.  Since $S\subseteq G[n]$, we have $a+S\subseteq G[n]$. Let $g\in G$,
 and let $d$ be a proper divisor of $n$. Since $S$ is  \round{n}, the set $\{s\in S: ds=g-da \}=\{x\in a + S: dx=g\}$ must be finite. Thus, $a+S$ is \round{n} as well.
\end{proof}

Our first theorem in this section shows that  \Zc{\Zar_G}s  $X$  are  precisely the translates of  \round{\mathfrak{m}(X)} sets. (See Definition \ref{definition:M(X)} for the number $\mathfrak{m}(X)$.)

\begin{theorem}
\label{corollary:about:atoms:being:translates:of:round:sets}
For a countably infinite subset $X$ of an abelian group $G$, the following conditions are equivalent:
\begin{itemize}
\item[(i)] $X$ is a \Zc{\Zar_G},
\item[(ii)] there exists $a\in G$ such that $X-a\in\TT(G)$,
\item[(iii)] for every $x\in X$, the set $X-x$ is \round{\mathfrak{m}(X)}.
\end{itemize}
\end{theorem}
\begin{proof}
(i)~$\to$~(ii) 
By Proposition \ref{general:properties:of:atoms}(ii),  $\CL_{\Zar_G} (X)=a+G[n]$ for some $a\in G$ and $n\in\mathbb{N}$. By Lemma \ref{exponent:of:G[n]:can:be:assumed:equal:to:n},  we may assume that $G[n]$ has exponent $n$. Let $S=X-a$. 
By Theorem \ref{translation:of:closures}(ii), $\CL_{\Zar_G} (S)=\CL_{\Zar_G} (X-a)=\CL_{\Zar_G} (X)-a=G[n]$.  Furthermore, since $X$ is a \Zc{\Zar_G}, from  Corollary \ref{continuity:of:some:operations:in:the:Zariski:topology}(ii) we conclude that $S$ must be \Zc{\Zar_G} as well. Applying the implication (ii)~$\to$~(i)  of Proposition  \ref{round:sets:as:dense:atoms}, we obtain that $S$ must be \round{n}. That is, $S\in\TT_n(G)\subseteq \TT(G)$.

(ii)~$\to$~(iii) 
Let $S=X-a$ and $n=\n_S$. By Corollary \ref{extension:of:function:n:to:all:subsets},
$\mathfrak{m}(S)=n$. Therefore, $\mathfrak{m}(X)=\mathfrak{m}(X-a)=\mathfrak{m}(S)=n$ since the function $\mathfrak{m}(-)$ is translation invariant. Let $x\in X$. Since $x-a\in G[n]$, we have $a-x\in G[n]$, and so $a-x+S=X-x$  is  \round{n}  by Lemma \ref{L_translates:of:n-round}.

(iii)~$\to$~(i) Let $x\in X$. Since $S=X-x$ is \round{\mathfrak{m}(X)}, it is a \Zc{\Zar_G}  
by the implication (i)~$\to$~(ii) of Proposition \ref{round:sets:as:dense:atoms}.
Then $X=x+S$ is a \Zc{\Zar_G} by Corollary \ref{continuity:of:some:operations:in:the:Zariski:topology}(ii).
\end{proof}

Corollary \ref{0-round:vs:n(-)=0} shows that  the translates in the above theorem are not necessary when $\mathfrak{m}(X)=0$.

We now characterize  {\em all\/} elements $b\in G$ such that $X-b$ is an \round{\mathfrak{m}(X)} set. 

\begin{proposition}\label{translate:of:n_round}
Let  $X$ be  a \Zc{\Zar_G} of an abelian group $G$,
 and let $n=\mathfrak{m}(X)$. Given $b\in G$, the set $X - b$ is  \round{n}  if and only if $b\in X + G[n]$. 
\end{proposition}

\begin{proof} Assume that $b\in X+ G[n]$. Then  there exists $x\in X$ such that $b-x \in G[n]$.  By the implication (i)~$\to$~(iii) of Theorem 
\ref{corollary:about:atoms:being:translates:of:round:sets}, $S_x=X - x$ is an \round{n} set. Since $x-b\in G[n]$, Lemma \ref{L_translates:of:n-round} implies that  $X-b= x - b + S_x $  is an \round{n} set.

Now assume that $X - b$ is an \round{n} set.   By implication (ii)~$\to$~(i) of Theorem \ref{corollary:about:atoms:being:translates:of:round:sets}, $X$ is a \Zc{\Zar_G}.
Choose any $x\in X$. Then  $S_x=X - x$ is an \round{n} set by the implication (i)~$\to$~(iii) of Theorem  \ref{corollary:about:atoms:being:translates:of:round:sets}.
Since $X-b= x-b+S_x$ is an \round{n} set, Lemma \ref{L_translates:of:n-round} implies that $b-x\in G[n]$. Therefore, $b\in x + G[n]\in X + G[n]$. 
\end{proof}

Now we characterize the non-trivial irreducible sets in the Zariski topology.
 
\begin{theorem}
\label{characterization:of:irreducible:sets}
Let $G$ be an abelian group and $X$ an infinite subset of $G$. Then the following conditions are equivalent:
\begin{itemize}
\item[(i)] $X$ is an irreducible subset of $(G,\Zar_G)$,
\item[(ii)] there exist $a\in G$ and $S\in\TT(G)$ such that $a+S$ is $\Zar_G$-dense in $X$,
\item[(iii)] there exist $n\in\mathbb{N}$ and $a\in X$ such that $X-a$ contains a $\Zar_G$-dense \round{n} subset of $G$.
\end{itemize}
\end{theorem} 

\begin{proof} 
(i)~$\to$~(iii) Since $X$ is irreducible, by Fact \ref{Fact:irreducible}(iii), $\CL_{\Zar_G}(X)$ is a $\Zar_G$-closed 
irreducible set. Therefore, $\CL_{\Zar_G}(X)=a+G[n]$ for some $a\in X$ and $n\in\mathbb{N}$. 
By Lemma \ref{exponent:of:G[n]:can:be:assumed:equal:to:n}, we may assume that $G[n]$ has exponent
$n$. Since translates are $\Zar_G$-homeomorphisms by Corollary \ref{continuity:of:some:operations:in:the:Zariski:topology}(ii),
$Y=X-a$ is an infinite $\Zar_G$-dense subset of $G[n]$, and $G[n]$ is irreducible.  From the implication (ii)~$\to$~(i) of Proposition \ref{finding:n-round:sets}, it 
follows that $Y$ contains an \round{n}  subset $S$. Since $G[n]=\CL_{\Zar_G} (S)\subseteq \CL_{\Zar_G}(Y)\subseteq G[n]$ by Proposition  \ref{round:sets:as:dense:atoms}, $S$ is $\Zar_G$-dense in $Y$. 

(iii)~$\to$~(ii) Let $S\subseteq X-a$ be an  \round{n} set that is  $\Zar_G$-dense in $X-a$.  Once again  applying
the fact that translates are $\Zar_G$-homeomorphisms, we conclude that $a+S$ must be $\Zar_G$-dense in $a+(X-a)=X$. 

(ii)~$\to$~(i) By Theorem \ref{corollary:about:atoms:being:translates:of:round:sets}, $a+S$  is a \Zc{\Zar_G}, 
and so it is irreducible. Since $a+S$ is $\Zar_G$-dense in $X$,  the latter set must be irreducible as well 
by Fact \ref{Fact:irreducible}(iii).
\end{proof}

\begin{corollary}\label{closure:irreducible:sets}
A subset of an abelian group is irreducible if and only if it contains a  $\Zar_G$-dense \Zc{\Zar_G}. 
\end{corollary}

\begin{proof}
From Theorem \ref{characterization:of:irreducible:sets}, we conclude that  a subset $X$ of an abelian group $G$ is irreducible if 
there exist 
$n\in\mathbb{N}$, 
$a\in G$ and an 
\round{n} set $S$ 
such that $a+S$ is $\Zar_G$-dense in $X$. By Proposition \ref{round:sets:as:dense:atoms}, $S$ is a \Zc{\Zar_G}.
Furthermore, $a+S$ is a \Zc{\Zar_G} by Corollary \ref{continuity:of:some:operations:in:the:Zariski:topology}(ii). 
\end{proof}

\begin{corollary}
\label{closures:of:atoms}
For a subset $F$ of an abelian group $G$, the following conditions are equivalent:
\begin{itemize}
  \item[(i)] $F$ is a closed, irreducible subset of $(G,\Zar_G)$,
  \item[(ii)] $F$ is a $\Zar_G$-irreducible elementary algebraic set,
  \item[(iii)] $F$ is a $\Zar_G$-connected elementary algebraic set,
  \item[(iv)] $F$ coincides with the $\Zar_G$-closure of some \Zc{\Zar_G}.
\end{itemize}
\end{corollary}
\begin{proof} The implication (i)~$\to$~(ii) is trivial.
The equivalence 
(ii)~$\leftrightarrow$~(iii)
follows from Corollary \ref{connected:vs:irreducible}. 
Finally, the implications (ii)~$\to$~(iv) and (iv)~$\to$~(i) follow from Corollary \ref{closure:irreducible:sets}.
\end{proof}

\section{Description of Zariski dense sets}
\label{section:7}

\begin{theorem}
\label{characterizing:Zariski:dense:sets}
Let $G$ be an infinite abelian group and $n=eo(G)$. Then for every subset $X$ of $G$,
 the following conditions are equivalent:
\begin{itemize}
  \item[(i)] $X$ is Zariski dense in $G$,
  \item[(ii)] $a+X$ contains an \round{n} set for every $a\in G$,
  \item[(iii)] there exist a finite set $F\subseteq G$ and a family $\{S_x: x\in F\}$ of  \round{n} sets such that $F+G[n]=G$ and $x+S_x\subseteq X$ for each $x\in F$.
\end{itemize}
\end{theorem}
\begin{proof} 
According to Lemma \ref{lemma:4.2a}(i) and Theorem \ref{essential:order:and:components}, $G[n]$ is an irreducible, $\Zar_G$-clopen subgroup of $G$ 
with a  finite index. 

(i)~$\to$~(ii). Fix $a\in G$. From 
(i) and Corollary \ref{continuity:of:some:operations:in:the:Zariski:topology}(ii), we conclude that $a+X$ is $\Zar_G$-dense in $G$.  Since $G[n]$ is $\Zar_G$-open,  $Y=(a+X)\cap G[n]$ must be $\Zar_G$-dense in $G[n]$. Since $G[n]$ is irreducible, we can apply Proposition \ref{finding:n-round:sets} to conclude that $Y$ must contain an \round{n} set.

(ii)~$\to$~(iii) 
This implication holds because $G[n]$ has  a finite index in $G$.

(iii)~$\to$~(i). Let $a\in F$. Since $a+X$ contains an \round{n} set, $G[n]\subseteq \CL_{\Zar_G}(a+X)$ by Proposition \ref{round:sets:as:dense:atoms}. Since $G[n]$ is $\Zar_G$-open in $G$, $X_a=(a+X)\cap G[n]$ is $\Zar_G$-dense in $G[n]$. Hence $-a+X_a$ is $\Zar_G$-dense in $-a+G[n]$ by Corollary \ref{continuity:of:some:operations:in:the:Zariski:topology}(ii). Since $-a+X_a\subseteq X$, it follows that $X\cap (-a+G[n])$ is $\Zar_G$-dense in $-a+G[n]$.

Note that $-F+G[n]=-(F+G[n])=G$, and so $\bigcup\{-a+G[n]:a\in F\}=G$.  Since each $X\cap (-a+G[n])$ is $\Zar_G$-dense in $-a+G[n]$, we conclude that $X$ is $\Zar_G$-dense in $G$.
\end{proof}

One may wonder if it is possible to strengthen  Theorem \ref{characterizing:Zariski:dense:sets}(iii)
  by requiring all $S_x$ to be equal to a {\em single\/} \round{n} set $S$. Since $\bigcup _{x\in F}(x+S_x)= F+S$ in this case, the modified item (iii)  then reads as
``$F+S\subseteq X$ for some \round{n} set $S$ of $G$''.  We will show in Theorem \ref{dense:in:irreducible} below that such a replacement is possible  precisely when $(G,\Zar_G) $ is irreducible. For the proof of  this  modified theorem,  we  need the following lemma: 

\begin{lemma}
\label{trimming:n-torsion:sets} Assume that $X$ is a countably infinite  subset of an abelian group $G$.
Then there exist  infinite disjoint sets $Y_0,Y_1\subseteq X$ such that $(a_0+Y_0)\cap (a_1+Y_1)$ is finite whenever $a_0,a_1\in G$.
\end{lemma}
\begin{proof} 
The smallest subgroup $H$ of $G$ containing $X$ is countable, so we can fix an enumeration $H=\{h_n:n\in\N\}$ of $H$. For $n\in\N$, define $F_n=\{h_0,\dots,h_n\}$. Since $X$ is infinite,  using a straightforward recursion on $n\in\N$ we can choose $x_n\in X\setminus (\{x_0,\dots,x_{n-1}\}+(F_n\cup -F_n))$. We claim that  $Y_0=\{x_{2n}:n\in\N\}$ and $Y_1=\{x_{2n+1}:n\in\N\}$ are as required
by Lemma \ref{trimming:n-torsion:sets}. Clearly, $Y_0$ and $Y_1$ are infinite disjoint
subsets of $X$.  Suppose that $a_0,a_1\in G$. Note that $(a_0+Y_0)\cap (a_1+Y_1)=a_0+(Y_0\cap (a_1-a_0+Y_1))$, so we may assume without loss of generality that $a_0=0$. If $Y_0\cap (a_1+Y_1)=\emptyset$, the proof is finished. So assume that $Y_0\cap (a_1+Y_1)\not=\emptyset$. Then $a_1\ne 0$  as $Y_0$ and $Y_1$ are disjoint, and $a_1\in Y_0-Y_1\subseteq X-X\subseteq H$. Thus,
there exists $n\in\N$ such that $a_1=h_n$. We claim that $Y_0\cap (a_1+Y_1)\subseteq \{x_0,\dots,x_{n}\}$. Indeed, suppose that $x_k\in Y_0\cap (a_1+Y_1)$ for some $k>n$. Then there exists some $l\in \N$ such that $x_k=a_1+x_l=h_n+x_l$. Clearly, $k\not=l$ 
since $a_1\ne 0$. If $k<l$, then $n<l$ holds as well, and so 
$x_l=x_k-h_n\in \{x_0,\dots,x_{l-1}\}-F_{l-1}$,  which is  a contradiction. If $l<k$, then $x_k=h_n+x_l\in F_{k-1}+\{x_0,\dots,x_{k-1}\}$, which is  again a contradiction.
\end{proof}

\begin{theorem}\label{dense:in:irreducible}
Let $G$ be a non-trivial abelian group and $n=eo(G)$. Then the following conditions are equivalent:
\begin{itemize}
    \item[(a)] $(G,\Zar_G) $ is irreducible, 
    \item[(b)]  for every $\Zar_G$-dense set $X$, there exist an  \round{n}  $S$ of  $G$ and a finite set $F$ with  $F+G[n]=G$ and $F+S \subseteq X$, 
    \item[(c)] for every $\Zar_G$-dense set $X$, there exist  an infinite set $Z\subseteq G[n]$ and a subset $F$ of $G$
such that  $F+G[n]=G$ and $F+Z \subseteq X$. 
\end{itemize}
\end{theorem}

\begin{proof}  (a)~$\to$~(b) Let $X$ be a $\Zar_G$-dense subset of $G$. Since $G=G[n]$ is irreducible by Corollary \ref{when:G:is:connected:and/or:irreducible},  
$X$ contains an   \round{n} set $S$ by Proposition \ref{finding:n-round:sets}.  Now  (b) holds when $F=\{0\}$.

(b) $\to $ (c) follows obviously from the definition of an   \round{n} set.

(c) $\to $ (a) Clearly, $G$ is infinite, and so by Corollary \ref{when:G:is:connected:and/or:irreducible}, it suffices to prove that $G=G[n]$.
Assume, by contradiction, that $G[n]\not=G$. Thus, we can choose $g\in G\setminus G[n]$. By Lemma \ref{lemma:4.2a}(ii), $eo(G[n])=eo(G)=n$. Thus, 
$G$ contains an \round{n} set $S$ by Corollary \ref{description:of:when:there:are:n-round:sets}. Let $Y_0$ and $Y_1$ be subsets of $S$ satisfying 
Lemma \ref{trimming:n-torsion:sets} with $S$  substituted by $X$. As an infinite subset of an \round{n} set, each $Y_i$ is \round{n}. Thus, both $Y_0$ and $Y_1$ are $\Zar_G$-dense in $G[n]$ by Proposition  \ref{round:sets:as:dense:atoms}. Hence, $g+Y_1$ is $\Zar_G$-dense in $g+G[n]$ by Theorem \ref{translation:of:closures}(ii). It follows that the set 
\begin{equation}
\label{set:X}
X=Y_0\cup (g+Y_1)\cup (G\setminus (G[n]\cup (g+G[n])))
\end{equation}
is $\Zar_G$-dense in $G$, and so (c) allows us to fix  a finite set  $F\subseteq G$ and an infinite set $Z\subseteq G[n]$ satisfying $F+G[n]=G$ and $F+Z \subseteq X$. Since $Y_0\cup Y_1\subseteq S\subseteq G[n]$ and $g\not\in G[n]$, from (\ref{set:X}) we get
\begin{equation}
\label{intersections:of:X}
X\cap G[n]=Y_0
\ \ \mbox{ and }\ \ 
X\cap (g+G[n])=Y_1.
\end{equation}

Since $0\in G=F+G[n]$, there exists some $f_0\in F\cap G[n]$. Since $f_0+Z\subseteq F+Z\subseteq X$, we have  $f_0+Z=(f_0+Z)\cap G[n]\subseteq X\cap G[n]=Y_0$ by (\ref{intersections:of:X}).  Therefore, $Z\subseteq -f_0+Y_0$.

Similarly, since $g\in G=F+G[n]$, there exists some $f_1\in F\cap (g+G[n])$. Since $f_1+Z\subseteq F+Z\subseteq X$, we have  $f_1+Z=(f_1+Z)\cap (g+G[n])\subseteq X\cap (g+G[n])=g+Y_1$ by (\ref{intersections:of:X}). Therefore, $Z\subseteq g-f_1+Y_1$.

We have proved that $Z\subseteq (-f_0+Y_0)\cap (g-f_1+Y_1)$. Since the latter set is finite by our choice of $Y_0$ and $Y_1$, we conclude that $Z$ must be finite as well, which is a contradiction.
\end{proof}

For an unbounded abelian group $G$, our next theorem provides a very simple characterization of Zariski dense subsets of $G$.

\begin{theorem}
\label{Zariski:dense:sets:in:unbounded:groups}
Let $G$ be an unbounded abelian group.  For an infinite subset $X$ of $G$, the following conditions are equivalent:
\begin{itemize}
\item[(i)] $X$ is $\Zar_G$-dense in $G$,
\item[(ii)] $mX$ is infinite for every  integer $m\in\mathbb{N}\setminus\{0\}$,
\item[(iii)] $X$ contains an \round{0} set.
\end{itemize}
\end{theorem}
\begin{proof}
First, note that $G[0]=G$.

(i)~$\to$~(iii)
By (i), we have $\CL_{\Zar_G}(X)= G=G[0]$. Since $G$ is unbounded, $eo(G[0])=eo(G)=0$ by Definition \ref{definition:of:essential:order}(ii). Applying the implication (iii)~$\to$~(i) of Proposition \ref{finding:n-round:sets}, we conclude that $X$ contains an \round{0} set. 

(iii)~$\to$~(i)
Let $S$ be an \round{0} subset of $X$. Now 
$G=G[0]=\CL_{\Zar_G}(S)\subseteq \CL_{\Zar_G}(X)\subseteq G$ by the implication (i)~$\to$~(ii) of  Proposition \ref{round:sets:as:dense:atoms}. Thus, $X$ is $\Zar_G$-dense in $G$.

The equivalence (ii)~$\leftrightarrow$~(iii) is proved in Proposition \ref{sets:that:contain:zero-round:sets}.
\end{proof}

\section{Description of the Zariski closure}
\label{Closure:section}

\begin{theorem}
\label{computing:the:closure}
Let $X$ be an infinite subset of an abelian group $G$, let $D$ be the finite set of  $\Zar_G$-isolated points of $X$, and let $X\setminus D=\bigcup_{i=1}^k X_i$ be the (unique) decomposition of $X\setminus D$ into irreducible components. 
Then there exist $a_1,\dots,a_k\in G$ and $S_1,\dots,S_k\in \TT(G)$ such that
\begin{itemize}
  \item[(i)] each $a_i+S_i$ is $\Zar_G$-dense in $X_i$, and
  \item[(ii)] 
$\displaystyle \CL_{\Zar_G}(X)=D\cup \bigcup_{i=1}^k (a_i +\CL_{\Zar_G}(S_i))=D\cup\bigcup_{i=1}^k (a_i+G[\n_{S_i}])$.
\end{itemize}
\end{theorem}

\begin{proof} Observe that $(X,\Zar_G\restriction_S)$ is a Noetherian space by Theorem 
\ref{Zariski.topology.is.Noetherian}(i) and   Fact \ref{Noetherian:facts}(1). Hence, $D$ is finite by Fact \ref{remark:irreducible:components}. According to the same  fact, each $X_i$ is infinite, so we can apply Theorem \ref{characterization:of:irreducible:sets} to find $a_i\in G$ and $S_i\in\TT(G)$ 
such that $a_i+S_i$ is $\Zar_G$-dense in $X_i$.  This yields (i). To prove (ii), note that 
\begin{equation}
\label{eq2}
\CL_{\Zar_G}(X_i)=\CL_{\Zar_G}(a_i+S_i)=a_i+\CL_{\Zar_G} (S_i)=a_i+G[\n_{S_i}]
\hbox{ for every } 
i\le k,
\end{equation}
 where the first equation follows from item (i),  the second equation  follows from Theorem \ref{translation:of:closures}(ii), and the
  third  equation follows from Proposition \ref{round:sets:as:dense:atoms}. Since $D$ is a finite  subset of a $T_1$-space, $D$ is closed in $(G,\Zar_G)$, and so
$$
\CL_{\Zar_G}(X)=\CL_{\Zar_G}\left(D\cup\bigcup_{i=1}^k X_i\right)=D\cup\bigcup_{i=1}^k \CL_{\Zar_G}(X_i),
$$
 which together with (\ref{eq2}) yields (ii).
\end{proof}

\begin{remark}\label{redundancy} 
\begin{itemize}
\item[(i)]
Let $X$, $G$ and $D$ be as in the assumption of Theorem \ref{computing:the:closure}. Define 
$$
I_X=\{(n,a)\in\mathbb{N}\times G:  X-a\mbox{ contains an \round{n} set}\}.
$$
We shall see now that the set $\mathscr{A}=\{a +G[n]:(n,a)\in I_X\}$ is finite, and 
\begin{equation}
\label{(*)}
\CL_{\Zar_G}(X)=D\cup \bigcup_{(n,a)\in I_X} (a +G[n])=D\cup \bigcup \mathscr{A}.
\end{equation}

Indeed, let $X\setminus D=\bigcup_{i=1}^k X_i$ be the decomposition of $X\setminus D$ into irreducible components. Let $a_1,\dots,a_k\in G$ and $S_1,\dots,S_k\in\TT(G)$
be as in the conclusion of  Theorem \ref{computing:the:closure}. Clearly, $(\n_{S_i},a_i)\in I_X$ 
for all $i$. This proves the inclusion $\CL_{\Zar_G}(X)\subseteq D\cup \bigcup_{(n,a)\in I_X} 
(a +G[n])$. To prove the reverse inclusion, suppose that $(n,a)\in I_X$. Then, $X-a$ contains some set  $S\in\TT_n(G)$, and so $a+G[n]=\CL_{\Zar_G}(a+S)\subseteq \CL_{\Zar_G}(S)$  by Proposition 
\ref{round:sets:as:dense:atoms}. Since the elementary algebraic set 
$a +G[n]\in \mathscr{A}$ contains the  dense irreducible set $a+S$,  it is irreducible as well, and the inclusion $a +G[n]\subseteq  \CL_{\Zar_G}(X)=D\cup\bigcup_{i=1}^k (a_i+G[\n_{S_i}])$ 
now yields  $a +G[n]\subseteq a_i+G[\n_{S_i}]$ for some $i$. This proves that each closed set $a +G[n]\in \mathscr{A}$ is contained in one of the finitely many closed sets $a_i+G[\n_{S_i}]$.  Since our space $(G,\Zar_G)$ is Noetherian, this yields that $\mathscr{A}$ is  finite. 

\item[(ii)]
The  union (\ref{(*)}) may be redundant; here is an easy example to this effect.  
Let $G=\Z(4)^{(\omega)}$, and let $S$ be the canonical base of $G$. Then $S$ is an  \round{4} set.  Let $X=S \cup 2 S$. Then $X$ is irreducible  (since it contains the dense
 irreducible set  $S$) and $\Zar_G$-dense, so $\CL_{\Zar_G}(X)=G=G[4]$ and the union from 
Theorem  
 \ref{computing:the:closure}(ii)  has a single member, namely 
 $\CL_{\Zar_G}(S)=\CL_{\Zar_G}(X)$. However, 
 $\mathscr{A}=\{G[2], G[4]\}$, so the term $G[2]$ in the union $G= G[2]\cup G[4]$ from (\ref{(*)}) is redundant. 

\item[(iii)] The union from  Theorem  \ref{computing:the:closure}(ii) is not redundant, since it gives the (unique)  decomposition
of the closed set $\CL_{\Zar_G}(X)$  into
a union of its irreducible components;
see Fact \ref{decomposition:into:irreducible:components}.
  For every $d\in D$, the singleton $\{d\}$ is an irreducible component of $\CL_{\Zar_G}(X)$, which means that  the set $D$ is disjoint from $\bigcup_{i=1}^k (a_i+G[\n_{S_i}])$.  
In particular, $D=\emptyset$ when $G$ is infinite, and $X$ is $\Zar_G$-dense in $G$.
\end{itemize}
\end{remark}

\begin{remark}
{\em If at least one of the integers  $\n_{S_i}$ appearing in Theorem \ref{computing:the:closure}(ii)
is equal to $0$, then $D=\emptyset$, $k=1$, and the union from Theorem \ref{computing:the:closure}(ii) contains only one member, namely, $G[0]=G$. 
}
Indeed, assume that  $\n_{S_i}=0$ for some $i\le k$. Then  $G[\n_{S_i}]=G[0]=G$, and so $G=a_i+G[\n_{S_i}]\subseteq \bigcup_{j=1}^k (a_j+G[\n_{S_j}])\subseteq G$.
Hence, $D=\emptyset$, and $k=i=1$ by Remark \ref{redundancy}(iii). 
\end{remark}

\begin{remark}
Let $X$ be a \Zc{\Zar_G} of an abelian group $G$. Then $S=X-x$ is an \round{\mathfrak{m}(X)} set for some (in fact, each)
 $x\in X$; see Theorem \ref{corollary:about:atoms:being:translates:of:round:sets}. Therefore,  $\CL_{\Zar_G}(X) = \CL_{\Zar_G}(x+S)=x+\CL_{\Zar_G}(S)=x+G[\mathfrak{m}(X)]$ by Theorem  \ref{translation:of:closures}(ii) and the implication (i)~$\to$~(ii) of Proposition \ref{round:sets:as:dense:atoms}. 
 Since the union from Theorem  \ref{computing:the:closure}(ii) is not redundant by Remark \ref{redundancy}(iii), we  conclude that the unique number  $\n_{S_i}$ appearing in Theorem 
 \ref{computing:the:closure}(ii) coincides with $\mathfrak{m}(X)$. This conclusion fails in  a more general case,  for example, when $X$ is an irreducible set. Indeed,  let us note that for the subset $X$ of the group $G$ from Remark \ref{redundancy}(ii) one has $\mathfrak{m}(X) = 2$, while $\CL_{\Zar_G}(X)=G[4]$.
\end{remark}

\begin{example} Here we describe the closure of an infinite subset $X$ of $G$ in the case when $X$ is a {\em subgroup\/}. 
\begin{itemize}
   \item[(a)] If $X$ is unbounded, then  $X$ contains an \round{0} set by Corollary \ref{groups:that:conatin:0-round:sets}. Hence, $X$ is $\Zar_G$-dense.
   \item[(b)] If $X$ is a bounded subgroup, then by Pr\" ufer's theorem $X$ is a direct sum of finite cyclic subgroups. 
   Hence, we can write $X=F \oplus X_1$, where $F$ is finite and  $X_1\cong \bigoplus_{i=1}^k \Z(n_i)^{(\alpha_i)}$ for 
   some infinite cardinals $\alpha_i$. In particular, $X_1$ contains a subgroup isomorphic to 
   $\Z(n)^{(\omega)}\cong \bigoplus_{i=1}^k \Z(n_i)^{(\omega)}$, where $n$ is the least common multiple of all $n_i$. By Proposition \ref{U-K:invar}(b), $X_1$ contains an  \round{n} set 
   $S$, and so $\CL_{\Zar_G}(X_1)=G[n]$ by Proposition \ref{round:sets:as:dense:atoms}. Being finite, $F$ is $\Zar_G$-closed,  and we conclude from Theorem \ref{translation:of:closures}(i) that  $\CL_{\Zar_G}(X)=\CL_{\Zar_G}(F)+ \CL_{\Zar_G}(X_1)=F+ G[n]$. 
   \item[(c)] It follows from (a) and (b) that a subgroup $X$ of $G$ is 
   $\Zar_G$-dense if and only if either $X$ is unbounded, or both $X$ and $G$ are bounded with $eo(X)=eo(G)=n$ and $nG=nX$.  
   \item[(d)] From (a) and (b), we deduce that the Zariski closure of a 
   subgroup is always a subgroup. This can also be deduced  directly from Theorem \ref{translation:of:closures}. 
Since $(G,\Zar_G)$ is a quasi-topological group (see the analysis following Corollary 
\ref{continuity:of:some:operations:in:the:Zariski:topology}), the same conclusion can be deduced from  
\cite[Proposition 1.4.13]{AT} as well.
\end{itemize}
\end{example}

In the rest of this section, we present the major corollaries of Theorem  \ref{computing:the:closure}. 
Our first corollary significantly strengthens \cite[Lemma 3.6]{DT}:

\begin{corollary}\label{dense:translates:of:torsion} Let $X$ be an infinite subset of an abelian group $G$. Then there exist $k\in\mathbb{N}$, 
$a_1,\dots,a_k\in G$ and $S_1,\dots,S_k\in\TT(G)$
such that $\bigcup_{i=1}^k (a_i+ S_i)$ is $\Zar_G$-dense in $X\setminus D$, where $D$ is the finite set of  $\Zar_G$-isolated points of $X$.
\end{corollary}

Our next proposition describes the case  in which  the union $\bigcup_{i=1}^k (a_i+ S_i)$ in Corollary \ref{dense:translates:of:torsion} 
is not only $\Zar_G$-dense in $X\setminus D$, but actually {\em coincides\/} with  $X\setminus D$.  

\begin{proposition}\label{characterization:dimension1}
Let $X$ be a countably infinite subset of an abelian group $G$. Then $\dim X =1$ if and only if  there exist a natural number $k$ 
and a \Zc{\Zar_G} $X_i$ for each $i\le k$, such that  $\bigcup_{i=1}^k X_i=X\setminus D$,  where $D$ is the finite set of  $\Zar_G$-isolated points of $S$.
\end{proposition}

\begin{proof} To prove the ``if'' part,  assume that sets $X_i$ with the desired properties are given.  According to Definition \ref{definition:of:Zariski:atom} and Fact \ref{FactZcurve}, $\dim X_i=1$. Then $\dim \left(\bigcup_{i=1}^m X_i\right)=1$ by Fact \ref{FactDim}(d), and this yields $\dim X =1$, as $\dim D=0$ by Fact  \ref{FactDim}(a).

To prove the ``only if'' part,  assume that $\dim X =1$. Let  $D$ be the finite set of  $\Zar_G$-isolated points of $X$, and let $X\setminus D=\bigcup_{i=1}^k X_i$  be the decomposition of $X\setminus D$ into  infinite irreducible components; see Fact \ref{remark:irreducible:components}. Fix $i=1,2,\ldots, k$. By Fact   \ref{FactDim}(b), $\dim X_i\le \dim X=1$.  Since $X_i$ is infinite, from   Fact  \ref{FactDim}(a) we derive the reverse inequality $\dim X_i\ge 1$. It follows that $X_i$ satisfies the assumptions of Fact \ref{FactZcurve}(a), and so $X_i$ is a \Zc{\Zar_G} by Definition \ref{definition:of:Zariski:atom}.
\end{proof}

\begin{corollary}
\label{hereditarily:separable}
For an abelian group $G$, the space $(G,\Zar_G)$ is hereditarily separable; that is, every subset $X$ of $(G,\Zar_G)$ has a countable dense subset.
\end{corollary}

\begin{proof} If $X$ is finite, then $X$ is its own countable dense subspace. Assume now that $X$ is infinite. Let $D$, $k$, $a_1,\dots,a_k$ and $S_1,\dots,S_k$ be as in Theorem \ref{computing:the:closure}. Then  $D\cup \bigcup_{i=1}^k (a_i+S_i)$ is a countable dense subset of $X$.
\end{proof}

\begin{corollary}
\label{Frechet-Urysohn}
For an abelian group $G$, the space $(G,\Zar_G)$ is Fr\'echet-Urysohn; that is, if $X$ is a subset of $G$ and $x\in \CL_{\Zar_G}(X)$, then there exists a sequence of points $\{x_j:j\in\mathbb{N}\}$ of $X$ converging to $x$.
\end{corollary}

\begin{proof}  If $x\in X$, by defining $x_j=x$ for all $j\in\mathbb{N}$, we  derive the required sequence. Therefore, from now on  we shall assume that $x\not \in X$; in particular 
$X$ is infinite. Let  $D$, $k$, $X_1,\dots, X_k$, $a_1,\dots,a_k$ and $S_1,\dots,S_k$ 
be as in Theorem \ref{computing:the:closure}. Since $x\in \CL_{\Zar_G}(X)\setminus X$, according to Theorem \ref{computing:the:closure}(ii), we conclude that  $x\not \in D$. Then $x\in \CL_{\Zar_G} (a_i+ S_i)$ for some $i\le k$.  By Theorem \ref{corollary:about:atoms:being:translates:of:round:sets}, $a_i+S_i$ is a \Zc{\Zar_G}. Let $\{y_j:j\in\mathbb{N}\}$ be any faithful enumeration of $a_i+S_i$. From Proposition  \ref{general:properties:of:atoms}(iv), we conclude that the sequence $\{y_j:j\in\mathbb{N}\}$ converges to $x$. Finally, note that $y_j\in a_i+S_i\subseteq X_i\subseteq X$ for every $j\in\mathbb{N}$.
\end{proof}

Item (b) of the next example demonstrates that Corollaries \ref{hereditarily:separable} and \ref{Frechet-Urysohn} are 
specific results about the space $(G,\Zar_G)$ for an abelian group $G$ that do not hold in general for  Noetherian $T_1$ spaces.

\begin{example}
\label{special:Noetherian:space}
\begin{itemize}
   \item[(a)] Let $X$ be a set, and let $\mathcal{T}_i$, $i=0,1,\ldots, n$, be Noetherian topologies on $X$.  Then   $\mathcal{T}=\sup_i \mathcal{T}_i$ is also a Noetherian topology on $X$.  Indeed, the family 
$
\mathscr{E}=\{\bigcap_{i=0}^n F_i:
X\setminus F_i\in\mathcal{T}_i \; \mbox{ for all }\; i\le n\}
$ 
satisfies the descending chain condition, $\mathscr{E}$ is closed under finite intersections, and $X\in\mathscr{E}$. Thus, by Fact \ref{building:Noetherian:spaces} we conclude that the family $\fin{\mathscr{E}}$ forms the family of closed sets of a unique topology $\mathcal{T}_{\mathscr{E}}$ on $X$ such that the space $(X, \mathcal{T}_{\mathscr{E}})$ is a Noetherian space. It remains only to note that $\mathcal{T}=\mathcal{T}_{\mathscr{E}}$.

    \item[(b)] We can use (a) to get an example of a Noetherian $T_1$-space that is neither separable nor  Fr\'echet-Urysohn. Indeed, let $\alpha>\omega_1$ be an ordinal, and let $X=\alpha$.  Let $\mathcal{T}_0=\{\{x\in X: \beta<x\}:\beta<\alpha\}\bigcup\{\emptyset,X\}$ be the upper topology of $X$, and  let $\mathcal{T}_1$ be the co-finite topology of $X$. Since both topologies are Noetherian, the topology $\mathcal{T}=\sup \{\mathcal{T}_0, \mathcal{T}_1\}$ is  Noetherian by (a). Since $\mathcal{T}_1$ is a $T_1$ topology, so is $\mathcal{T}$. Since $(X, \mathcal{T}_0)$ is not separable,  $(X, \mathcal{T})$ is not separable either. To show that  $(X, \mathcal{T})$ is not Fr\'echet-Urysohn, observe that $\omega_1\in X$ is in the $\mathcal{T}$-closure of the set $S=\{\gamma\in X: \gamma<\omega_1\}$, yet no sequence of points of $S$ converges to $\omega_1$.
\end{itemize}
\end{example}

Our last corollary of Theorem \ref{computing:the:closure} is used in the proof of Theorem \ref{Main:theorem}.

\begin{corollary}
\label{technique:for:realizing:Zariski:closure}
Let $X$, $G$, $D$, $k$ and $S_1,\dots,S_k\in\TT(G)$
be as in Theorem \ref{computing:the:closure}. If $\mathcal{T}$ is a Hausdorff group topology on $G$ such that $\CL_{\mathcal{T}}(S_i)=G[\n_{S_i}]$ 
for every  $i\le k$, then $\CL_{\mathcal{T}}(X)=\CL_{\Zar_G}(X)$.
\end{corollary}

\begin{proof} By the assumption of our corollary, 
$a_i+G[\n_{S_i}]=a_i+\CL_{\mathcal{T}}(S_i)=\CL_{\mathcal{T}}(a_i+S_i)\subseteq \CL_{\mathcal{T}}(S_i)$  for every $i\le k$. Combining this with Theorem \ref{computing:the:closure}(ii) yields
$$
\CL_{\Zar_G}(X) = D\cup\bigcup_{i=1}^k (a_i+G[\n_{S_i}])\subseteq D\cup \bigcup_{i=1}^k \CL_{\mathcal{T}}(S_i)=\CL_{\mathcal{T}}(X).
$$
The reverse inclusion $\CL_{\mathcal{T}}(X)\subseteq \CL_{\Zar_G}(X)$ follows from $\Zar_G\subseteq\Mar_G\subseteq \mathcal{T}$.
\end{proof}

\section{A precompact metric group topology realizing the Zariski closure}
\label{realizing:closure:section}

The main purpose of this section is to prove Theorem \ref{Main:theorem}. The proof of Theorem C, provided in the end of this section, then follows easily from Theorem \ref{Main:theorem}.

We start with two lemmas from \cite{DS} that will be needed in  our proofs. 

\begin{lemma}\label{lemma:Forcing_paper}
{\rm (\cite[Lemma 4.10]{DS})}
Suppose that $\mathscr{E}$ is a countable family of subsets of an abelian group $G$, $g\in G$ and $g\not=0$. Then there exists a group homomorphism $h:G\to\T$ such that: 
\begin{itemize}
\item[(i)]
 $h(g)\neq 0$,
\item[(ii)]
 if $n\in\N$, $z\in\T[n]$, $l\in\N\setminus\{0\}$ and $E\in\mathscr{E}\cap\TT_n(G)$,  then the set  $\{x\in E: |h(x)-z|<1/l\}$ is infinite.
\end{itemize}
\end{lemma}

\begin{lemma}{\rm (\cite[Lemma 3.17]{DS})}\label{lemma1:Forcing_paper}
Let $G$ and $H$ be Abelian groups such that $|G|\le r(H)$ and $|G|\le r_p(H)$ for each $p\in\Prm$. Suppose also that $G'$ a subgroup of $G$ such that
$r(G')<r(H)$ and $r_p(G')<r_p(H)$ for all $p\in\Prm$. If  $H$ is divisible,  then for every monomorphism $h:G'\to H$ there exists a monomorphism $\pi: G\to H$ such that $\pi\restriction_{G'}=h$.
\end{lemma}

\begin{lemma}
\label{lemma:metrizable:case}
Let $G$ be a countable abelian group and $\mathscr{S}$ be a countable subfamily of $\TT(G)$.
Then there exists a monomorphism  $h: G\to \T^\N$ satisfying the following property: if $S\in\mathscr{S}$,
$O$ is an open subset of $\T^\N$ and $O\cap \Tp[\n_S]^\N\not=\emptyset$, 
then the set $\{x\in S: h(x)\in O\}$ is infinite. In particular, $h(S)$ is dense in $\Tp[\n_S]^\N$ 
for every $S\in\mathscr{S}$.
\end{lemma}

\begin{proof} Let $\mathscr{S}=\{S_j:j\in\N\}$ be an enumeration of $\mathscr{S}$. For typographical reasons, let $n_j=\n_{S_j}$ 
for each $j\in \N$. Let $\{g_j:j\in\mathbb{N}\}$ be an enumeration of the countable set $G\setminus\{0\}$ and $\B$ be a  countable base of $\T$ with $\emptyset\not\in\B$. Define $\E_{-1}=\mathscr{S}$. By induction on $k\in\mathbb{N}$,  we construct a countable family $\E_k$ of subsets of $G$ and a homomorphism $h_k:G\to\T$  with the following properties:
\begin{itemize}
\item[(i$_k$)] $h_k(g_k)\not=0$,
\item[(ii$_k$)] $\E_{k-1}\subseteq \E_k$,
\item[(iii$_k$)] if $E\in\E_{k-1}\cap \TT(G)$,
$U\in\B$ and $U\cap \Tp[\n_E]\not=\emptyset$, 
 then the set $S_{E,U,k}=\{x\in E: h_k(x)\in U\}$ is infinite, 
\item[(iv$_k$)] if $E\in \E_{k-1}$, $U\in\B$ and the set $S_{E,U,k}$ is infinite, then $S_{E,U,k}\in\E_k$.
\end{itemize}
We apply Lemma \ref{lemma:Forcing_paper} to find a homomorphism $h_0: G\to\T$ satisfying (i$_0$) and (iii$_0$).  Define $\E_0=\E_{-1}\cup\{S_{E,U,0}: E\in\E_{-1}, U\in\B,   S_{E,U,0}$ is infinite$\}$.  Then (ii$_0$) and (iv$_0$) hold trivially. This completes the basis of induction. Suppose now that  $k\in\mathbb{N}\setminus\{0\}$ and for every $j<k$ we have already constructed a countable family $\E_j$ of subsets of $G$ and a homomorphism $h_j:G\to\T$ satisfying (i$_j$)--(iv$_j$). 
Let us define a countable family $\E_k$ of subsets of $G$ and a homomorphism $h_k:G\to\T$ satisfying 
(i$_k$)--(iv$_k$). We apply Lemma \ref{lemma:Forcing_paper} once again to find a homomorphism $h_k: G\to\T$ satisfying (i$_k$) and (iii$_k$). Define $\E_k=\E_{k-1}\cup\{S_{E,U,k}: E\in\E_{k-1}, U\in\B,  S_{E,U,k}$ is infinite$\}$ and note that (ii$_k$) and (iv$_k$) hold.

Define $h:G\to \T^\N$ by $h(x)=\{h_n(x)\}_{n\in\mathbb{N}}$ for $x\in G$.  Clearly, $h$ is a group homomorphism. If $g\in G\setminus \{0\}$, then $g=g_k$ for some $k\in\mathbb{N}$, and so  $h_k(g_k)\not=0$ by (i$_k$), which yields $h(g)=h(g_k)\not=0$. Hence $h$ is a monomorphism.

Assume that $j\in\mathbb{N}$, and $O$ is an open subset of $\T^\N$ such that $O\cap \Tp[n_j]^\N\not=\emptyset$. 
We shall prove that the set $\{x\in S_j: h(x)\in O\}$ is infinite. There exists $k\in\mathbb{N}$
and a sequence $U_0,\ldots,U_k$ of elements of $\B$ such that 
$U_0\times\cdots\times U_{k}\times \T^{\N\setminus \{0,\dots,k\}}\subseteq O$
  and $U_i\cap \Tp[n_j]\not=\emptyset$ for all $i\le k$.
From (iii$_0$) and  our definition  of $\E_{-1}$, it follows that $E_0=S_{S_j,U_0,0}$ 
 is infinite, and (iv$_0$) gives $E_0\in\E_0$.  As an infinite subset of the  \round{n_j} set $S_j$, the set 
  $E_0$ is also \round{n_j} by Remark \ref{Background}(iii), and so (iii$_1$) implies that $E_1=S_{E_0,U_1,1}$ is infinite. Thus,
   $E_1\in\E_1$ by (iv$_1$). Continuing this argument, we conclude that each set
$E_i=S_{E_{i-1},U_i,i}$ 
is    infinite for $i\le k$.  By construction, $E_{k}\subseteq E_{k-1}\subseteq \dots\subseteq E_0\subseteq S_j$. Finally, note that $E_{k}$ 
is an infinite subset of $S_j$ such that $h_i(x)\in  U_i$ whenever $x\in E_{k}$ and $i\le k$.
Therefore, $h(E_{k})\subseteq U_0\times\cdots\times U_{k}\times   \T^{\N\setminus \{0,\dots,k\}}\subseteq O$, which implies $E_{k}\subseteq \{x\in S_j: h(x)\in O\}$. 
Since the former set is infinite,  so is the latter.
\end{proof}

\begin{corollary}
A countably infinite Abelian group $G$  is isomorphic to a dense subgroup of  $\T^\N$ if and only if  $G$  is unbounded.
\end{corollary}

\begin{proof} Let $j:G \to \T^\N$ be a monomorphism such that $j(G)$ is a dense subgroup of  $\T^\N$. Then $nj(G)$ must be dense in $\T^\N = n\T^\N$, 
so $nG \ne \{0\}$ for every positive $n\in \N$. Therefore, $G$  is unbounded.
 
If $G$  is unbounded, then $G$ contains an \round{0} set $S$; see Corollary \ref{groups:that:conatin:0-round:sets}. By Lemma \ref{lemma:metrizable:case} applied to $\mathscr{S}=\{S\}$, there exists a monomorphism $j:G \to \T^\N$ such that $j(S)$ is dense in $\T^\N$. So $j(G)$  is dense in $\T^\N$ as well.  
\end{proof}

\begin{theorem}\label{Main:theorem}
Let $G$ be an abelian group with $|G|\leq \cont$, and let $\mathscr{X}$ be a countable family of subsets of $G$. Then there exists a precompact metric group topology
$\mathcal{T}$ on $G$ such that the $\mathcal{T}$-closure of each $X\in\mathscr{X}$ coincides with its $\Zar_G$-closure.
\end{theorem}

\begin{proof} Applying Corollary \ref{technique:for:realizing:Zariski:closure},  for every $X\in\mathscr{X}$ we can fix a finite family $\mathscr{S}_X\subseteq \TT(G)$ with the following property:
\begin{itemize}
\item[($\dagger_X$)] 
if $\mathcal{T}$ is a Hausdorff group topology on $G$ such that $\CL_{\mathcal{T}}(S)=G[\n_S]$  for each $S\in\mathscr{S}_X$, then  $\CL_{\mathcal{T}}(X)=\CL_{\Zar_G}(X)$.
\end{itemize}

Define $\mathscr{S}=\bigcup\{\mathscr{S}_X:X\in\mathscr{X}\}$. Let $G'$ be the countable subgroup of $G$ generated by  $\bigcup\mathscr{S}$. 
Since $\mathscr{S}\subseteq \TT(G)$ and $S\subseteq G'$ for each $S\in\mathscr{S}$, from Remark \ref{Background}(iv) we conclude that $\mathscr{S}\subseteq \TT(G')$. Therefore,
we can  apply Lemma \ref{lemma:metrizable:case} to find a monomorphism  $h:G'\to\T^\N$ such that  $h(S)$ is dense in $\Tp[\n_S]^\N$ 
for every $S\in\mathscr{S}$. Since $|\T^\N|=r(\T^\N)=\cont$, $r_p(\T^\N)=\cont$ for all $p\in\Prm$  (\cite{Fuchs}; see also  \cite[Lemma 4.1]{DS3}),  $|G|\le 2^\cont$, $|G'|=|\N|<\cont$, and $\T^\N$ is divisible, Lemma \ref{lemma1:Forcing_paper} allows us to find a monomorphism $\pi:G\to \T^\N$ extending $h$.
Denote by  $\Top$ the precompact metric group topology induced on $\pi(G)$ by the usual topology of $\T^\N$. 

Let $X\in\mathscr{X}$. Choose $S\in\mathscr{S}_X$ arbitrarily. Since  $\pi(G)[\n_S]=\pi(G)\cap \Tp[\n_S]^\N$ and $\pi(S)=h(S)$ is dense in $\Tp[\n_S]^\N$, it follows that $\pi(S)$ is $\mathcal{T}$-dense in $\pi(G)[\n_S]$. Identifying $G$ with $\pi(G)$, we conclude that $S$ is $\mathcal{T}$-dense in $G[\n_S]$. In particular, $\CL_{\mathcal{T}}(S)=\CL_{\mathcal{T}}(G[\n_S])$. Since $G[\n_S]$ is $\Zar_G$-closed in $G$ and $\Zar_G\subseteq \mathcal{T}$, it follows that  $\CL_{\mathcal{T}}(G[\n_S])=G[\n_S]$. We proved that 
$\CL_{\mathcal{T}}(S)=G[\n_S]$ for every $S\in \mathscr{S}_X$. Now ($\dagger_X$)  yields $\CL_{\mathcal{T}}(X)=\CL_{\Zar_G}(X)$.
\end{proof}

\begin{remark}
\label{closure:of:many:sets:cannot:be:realized}
Let $G$ be an infinite abelian group. 
\begin{itemize}
\item[(i)] {\em There exists a family $\mathscr{X}$ of subsets of $G$ such that $|\mathscr{X}|=\cont$, and for every 
Hausdorff group topology $\mathcal{T}$ on $G$, the $\mathcal{T}$-closure of  some $X\in\mathscr{X}$
differs from its $\Zar_G$-closure.\/} Indeed, let $H$ be a countably infinite subgroup of $G$ and $\mathscr{X}$ be the family 
of all subsets of $H$. Clearly, $|\mathscr{X}|=\cont$. Assume, by contradiction, that $\mathcal{T}$ is a Hausdorff group 
topology on $G$ such that $\CL_{\mathcal{T}}(X)=\CL_{\Zar_G}(X)$ for every $X\in\mathscr{X}$.  For $X\in\mathscr{X}$, we have 
$\CL_{\mathcal{T}\restriction_H}(X)=H\cap \CL_{\mathcal{T}}(X)=H\cap \CL_{\Zar_G}(X)=\CL_{\Zar_H}(X)$ by Lemma 
\ref{lemma:hereditary:Zariski}, and since $\mathscr{X}$ is the family of all subsets of $H$, we conclude
that $\mathcal{T}\restriction_H=\Zar_H$. Since $\mathcal{T}$ is a Hausdorff topology on $G$, the subspace topology 
$\mathcal{T}\restriction_H$ is Hausdorff as well. This contradicts Corollary \ref{Zariski:topology:never:group}.

\item[(ii)] Item (i) shows that the conclusion of Theorem \ref{Main:theorem} may fail for families  $\mathscr{X}$ of size 
 $\cont$. Under the Continuum Hypothesis, the conclusion of Theorem \ref{Main:theorem} would then fail for families  $\mathscr{X}$ of size $\omega_1$. Thus, in
 general, a single Hausdorff group topology on an abelian group $G$ cannot realize the Zariski closure of uncountably many
  subsets of $G$, that is, the assumption of countability for the family $\mathscr{X}$ in  Theorem \ref{Main:theorem} is necessary.
\end{itemize}
\end{remark}

\medskip
\noindent
{\bf Proof of Theorem C.\/}
The implication (i)~$\to$~(iii) is proved in Theorem \ref{Main:theorem}. The implication (iii)~$\to$~(ii) is trivial. Finally, the implication (ii)~$\to$~(i) holds because a precompact metric group has size at most $\cont$ (this follows from the well known fact that infinite compact metrizable group has size $\cont$ \cite[Corollary 5.2.7 b)]{AT}).
 \QED

\section{Proof of the equality $\Zar_G=\Mar_G={\mathfrak P}_G$ for an abelian group $G$}
\label{section:10}

For an abelian group $H$ we denote by $H^*$ the closed subgroup of $\T^H$ consisting of all  homomorphisms $\chi:H\to\T$; in particular, $H^*$ is compact in the subspace topology  inherited from $\T^H$. A subgroup $N$ of $H^*$ is called {\em point separating\/} if
for every $x\in H\setminus\{0\}$, there exists $\chi \in N$ such that $\chi(x) \ne 0$.

For every subgroup $N$ of $H^*$, let $e_{N,H}:H \to \T^N$ be the map defined by  $e_{N,H}(x)(\chi)=\chi(x)$ for $\chi\in N$ and each $x\in H$, and let $\mathscr{T}_{N,H}$ denote the coarsest topology on $H$ with respect to which $e_{N,H}$ becomes continuous. The straightforward proof of the following fact is left to the reader.

\begin{fact}
\label{easy:fact}
For an abelian group $H$ and a subgroup $N$ of $H^*$, the following conditions are equivalent:
\begin{itemize}
\item[(i)]
$N$ is point separating;
\item[(ii)]
$\mathscr{T}_{N,H}$ is a precompact Hausdorff group topology on $H$;
\item[(iii)] 
$e_{N,H}$ is a monomorphism;
\item[(iv)] 
the map $e_{N,H}:(H,\mathscr{T}_{N,H}) \hookrightarrow \T^N$ is a topological isomorphism between $(H,\mathscr{T}_{N,H})$ and the subgroup $e_{N,H}(H)$ of $\T^H$.
\end{itemize}
\end{fact}

The relevant fact that the correspondence  $N \mapsto \mathscr{T}_{N,H}$ between point separating subgroups of $H^*$ and precompact Hausdorff group topologies on $H$ is a bijection was pointed out by Comfort and Ross \cite{CRoss}.

Let $H$ be a normal subgroup of a group $G$ and $\mathcal{T}'$ a Hausdorff group topology on $H$. In general, it is impossible 
to find a Hausdorff group topology $\mathcal{T}$ on $G$ that induces the original topology $\mathcal{T}'$ on $H$. There are severe obstacles to the extension of group topologies even when the subgroup $H$  is abelian \cite{DS_JGT, DS_Zariski_embeddings}. 
However, in the case when the subgroup $H$ is central (in particular, when the group $G$ itself is abelian), the extension  problem becomes trivial. 
One can simply take the family of all $\mathcal{T}'$-neighborhoods of $0$ as a base of the family of  $\mathcal{T}$-neighborhoods of $0$ of a group topology $\mathcal{T}$ on 
 $G$. This topology $\mathcal{T}$ is  Hausdorff and obviously induces the original topology $\mathcal{T}'$ on $H$. Note that there may exist other group topologies $\mathcal{T}''$ on $G$ with $\mathcal{T}''\restriction_H=\mathcal{T}'$. 

The extension of {\em precompact\/} Hausdorff group topologies from a subgroup of an abelian group to the whole group is a bit more delicate. 

\begin{theorem}
\label{extension:of:precompact:group:topologies}
Let $H$ be a subgroup of an abelian group  $G$ and $\mathcal{T}'$  a precompact Hausdorff group topology on $H$. Then there 
exists a precompact Hausdorff group topology $\mathcal{T}$ on $G$ that induces $\mathcal{T}'$ on $H$; that is, $\mathcal{T}\restriction_H=\mathcal{T}'$ holds. 
\end{theorem}

\begin{proof}  Let $\rho: G^*\to H^*$ be the restriction homomorphism defined by  $\rho(\chi)=\chi \restriction_H$ for every $\chi \in G^*$.  Since $\T$ is a divisible group, for every $\chi\in H^*$, there exists an extension $\widetilde{\chi}\in G^*$ of $\chi$. Therefore, $\rho$ is a surjection.  It  
follows from Peter-Weyl's theorem that the subgroup $N$ of $H^*$ consisting of all $\mathcal{T}'$-continuous characters is point separating and $\mathcal{T}'=\mathscr{T}_{N,H}$; see \cite{CRoss}.

We claim that $A=\rho^{-1}(N)$ is a point separating subgroup  of $G^*$.  Indeed, let $g\in G\setminus \{0\}$. We need to find $\varphi\in A$ such that $\varphi(g)\ne 0$. If $g\in H$, then $\chi(g)\not=0$ for some $\chi\in N$, as  $N$ is a point separating subgroup of $H^*$. Since $\rho$ is a surjection, there exists $\varphi \in G^*$ such that  $\rho(\varphi) = \chi$. In particular, $\varphi\in \rho^{-1}(N)=A$. Furthermore, $\varphi(g)=\varphi\restriction_H(g)=\chi(g)\not=0$. Suppose now that $g\in G\setminus H$. Let $\pi:G\to G/H$ be the canonical quotient homomorphism. There exists a character  $\psi:G/H\to \T$ such that $\psi(\pi(g))\not=0$. Clearly, $\varphi=\psi\circ \pi\in G^*$ and $\rho(\varphi)=\varphi\restriction_H=0\in N$, which yields $\varphi\in \rho^{-1}(N)=A$. Finally, note that $\varphi(g)\not=0$.

Since $A$ is a point separating subgroup of $G$, it follows from Fact \ref{easy:fact}  that  $\mathcal{T}= \mathscr{T}_{A,G}$ is a precompact Hausdorff group topology on $G$.
It remains only to show that $\mathcal{T}\restriction_H=\mathcal{T}'$.  In other words, we aim to show that the inclusion map 
$j:(H,\mathscr{T}_{N,H})\hookrightarrow (G,\mathscr{T}_{A,G})$ is a  topological group embedding. 
 
For a set $Y$ and an element $y\in Y$, let $p_{Y,y}:\T^Y\to \T$ be the canonical projection defined by $p_{Y,y}(f)=f(y)$ for every $y\in Y$.
Since $A=\rho^{-1}(N)$ and $\rho$ is a surjection, $\rho\restriction_A: A \to N$ is surjective as well. Therefore, the  the map $\iota: \T^N \to \T^A$ defined by $\iota(f)= f\circ\rho\restriction_A$ for $f\in \T^N$, is a monomorphism. Furthermore, the map $\iota$ is continuous,  as  $p_{A,\chi} \circ \iota = p_{N,\rho(\chi)}$ is continuous for every  $\chi\in A$.
By the compactness of $\T^N$, we conclude that $\iota: \T^N \hookrightarrow \T^A$ is a topological group embedding. 

By  Fact \ref{easy:fact}, both maps $e_{A,G}:(G,\mathscr{T}_{A,G}) \hookrightarrow \T^A$ and $e_{N,H}:(H,\mathscr{T}_{N,H}) \hookrightarrow \T^N$ are topological group embeddings. Since the diagram
$$   \xymatrix{
(G,\mathscr{T}_{A,G}) \ar@{^{(}->}[r]^{\;\;\;\;\;\;e_{A,G}}    &  \T^A \\
       (H,\mathscr{T}_{N,H})\ar@{^{(}->}[r]^{\;\;\;\;\;\;\;\; e_{N,H}}\ar@{^{(}->}[u]^{j} & \T^N \ar@{^{(}->}[u]^{\iota}
       }$$
is commutative, we conclude that $j$ is a topological group embedding as well.
\end{proof}

\begin{remark} 
\begin{itemize}
     \item[(i)] In connection with  Fact \ref{easy:fact}, one should mention  the following important  result of  Comfort and Ross {\rm \cite{CRoss}}:  {\em Given an abelian group $H$, a subgroup $N$ of $H^*$ is point separating if and only if $N$ is dense in $H^*$\/}. 
      \item[(ii)] 
With the help of   the equivalence in (i), one can offer an alternative argument showing that the subgroup $A$ of $G^*$  in the proof of Theorem  \ref{extension:of:precompact:group:topologies} is point separating. Indeed,  since $G^*$ is a compact group and $\rho: G^*\to H^*$ is a continuous surjective 
 homomorphism, $\rho$ is an open map. Since $N$ is dense in $H^*$ by item (i),  this implies that $A=\rho^{-1}(N)$ is dense in $G^*$. Applying (i) once again, we conclude that $A$ is a point separating subgroup of $G^*$.
\end{itemize}
 \end{remark}

\begin{corollary}\label{hereditary:precompact:Markov}
If $H$ is a subgroup of an abelian group $G$, then ${\mathfrak P}_G\restriction_H\subseteq {\mathfrak P}_H$.
\end{corollary}

\begin{proof} Let $X$ be a subset of $H$ that is not ${\mathfrak P}_H$-closed. There exists a precompact Hausdorff group topology $\Top'$ on $H$ such that $X$ is not $\Top'$-closed. By Theorem \ref{extension:of:precompact:group:topologies}, there exists a Hausdorff group topology $\Top$ on $G$ such that $\Top\restriction_H=\Top'$. Then $X$ is not $\Top$-closed, and so $X$ cannot be ${\mathfrak P}_G$-closed either. Since $X\subseteq H$, we conclude that $X$ is not ${\mathfrak P}_G\restriction_H$-closed.
\end{proof}

\medskip
\noindent
{\bf Proof of Theorem A.\/}
Since $\Zar_G\subseteq \Mar_G\subseteq {\mathfrak P}_G$, it suffices to prove that ${\mathfrak P}_G\subseteq \Zar_G$. Let us note that this follows from Theorem \ref{Main:theorem} when $|G| \leq \cont$. Indeed, let $F$ be a subset of $G$ that is not $\Zar_G$-closed.
Applying Theorem \ref{Main:theorem} with $\mathscr{X}=\{F\}$, we can find a precompact metric group topology $\mathcal{T}$ on $G$ such that $\CL_{\mathcal T}(F)=\CL_{\Zar_G}(F)\ne F$; that is, $F$ is not $\mathcal{T}$-closed. Hence, $F$ cannot be ${\mathfrak P}_G$-closed either.

To prove that ${\mathfrak P}_G\subseteq \Zar_G$ in the general case, it suffices to pick an arbitrary subset $F$ of $G$ and check that $\CL_{\Zar_G}(F)\subseteq \CL_{{\mathfrak P}_G}(F)$. According to Corollary \ref{hereditarily:separable},  there exists a countable subset $X$ of $F$ such that   $F\subseteq \CL_{\Zar_G}(X)$. In particular,
$\CL_{\Zar_G}(F)\subseteq \CL_{\Zar_G}(X)$. Since $\CL_{{\mathfrak P}_G}(X)\subseteq \CL_{{\mathfrak P}_G}(F)$, it remains to check that $\CL_{\Zar_G}(X)\subseteq \CL_{{\mathfrak P}_G}(X)$.  

Fix an arbitrary $x\in \CL_{\Zar_G}(X)$, and let $H$ be the (countable) subgroup of $G$ generated by $X$ and $x$.   
By Lemma \ref{lemma:hereditary:Zariski},  
$x\in H\cap \CL_{\Zar_G}(X)=\CL_{\Zar_G\restriction_H}(X)=\CL_{\Zar_H}(X)$. 
By the initial part of the argument, $\Zar_H={\mathfrak P}_H$, as $H$ is countable. So 
$x\in \CL_{\Zar_H}(X)=\CL_{{\mathfrak P}_H}(X)$. Since $X\subseteq H$, Corollary \ref{hereditary:precompact:Markov} yields $\CL_{{\mathfrak P}_H}(X)\subseteq \CL_{{\mathfrak P}_G\restriction_H}(X)=H\cap \CL_{{\mathfrak P}_G}(X)$. Therefore, $x\in \CL_{{\mathfrak P}_G}(X)$. This proves that $\CL_{\Zar_G}(X)\subseteq \CL_{{\mathfrak P}_G}(X)$. \QED

\begin{corollary}\label{Perelman}
For an arbitrary subset $X$ of an abelian group $G$, the following conditions are equivalent:
 \begin{itemize}
  \item[(a)] $X$ is unconditionally closed,
  \item[(b)] $X$ is closed in every precompact Hausdorff group topology on $G$, 
  \item[(c)] $X$ is algebraic.
\end{itemize}
\end{corollary} The equivalence of (a) and (c) was attributed by Markov \cite{Mar} to Perel$'$man, though the proof never appeared in print. Recently, a proof 
of this equivalence was provided in \cite{DS_JGT}; see also \cite[Theorem 3.13]{TY}
for almost torsion-free abelian groups  and \cite[Proposition 4.6]{TY} for abelian groups of prime exponent. The group topologies involved in both results in \cite{TY} are {\em not\/} precompact, so these results do not include also the equivalence with (b) even in those two particular cases.  

\begin{remark} According to Lemma \ref{lemma:hereditary:Zariski} and  Corollary 
\ref{hereditary:precompact:Markov}, one might first study the spaces $(G,\Zar_G)$ and $(G,\mathfrak{P}_G)$ for {\em divisible  groups\/} $G$ and then use  them to obtain information on their subgroups. The advantage of having a divisible group $G$ is that, for every prime number $p$ and each integer $n\in\mathbb{N}$, the subgroup $G[p^n]$ of $G$ is always irreducible whenever it is infinite; see Example \ref{divisible:example:of:irreducible:subgroups}.
\end{remark}

\section{A partial solution to a problem of Markov} 
\label{section:11}

\noindent
{\bf Proof of Theorem D.\/}
(i)~$\to$~(iii)
 Assume that $X$ is potentially dense in $G$, and let $\mathcal{T}$ be a Hausdorff group topology on $G$ such that $X$ is $\mathcal{T}$-dense in $G$. Then $G=\CL_{\mathcal{T}}(X)\subseteq \CL_{\Zar_G}(X)\subseteq G$, which  yields $G=\CL_{\Zar_G}(X)$. 

(iii)~$\to$~(ii) This follows from Theorem \ref{Main:theorem}.

(ii)~$\to$~(i) This is trivial.
\QED

Combining Theorems D and \ref{characterizing:Zariski:dense:sets},  we obtain a partial solution 
to Markov's problem regarding  potentially dense sets for all abelian groups of size at most continuum:
  
\begin{corollary}\label{Solution:Markov:problem}
Let  $X$ be an infinite  subset of be an abelian group $G$ of size $\leq \cont$. Define $n=eo(G)$. Then the following conditions are equivalent:
\begin{itemize}
  \item[(i)] $X$ is potentially dense in $G$,
  \item[(ii)] $a+X$ contains an \round{n} set for every $a\in G$,
  \item[(iii)] there exist a finite set $F\subseteq G$ and a family $\{S_x: x\in F\}$ of  \round{n} sets such that $F+G[n]=G$ and $x+S_x\subseteq X$ for each $x\in F$.
\end{itemize}
\end{corollary}

Recall that  every infinite subset of an abelian group $G$ is Zariski dense if and only if $G$ is either almost torsion-free or has a prime exponent  (see Fact \ref{necessity:almost_torsion-free}). 

\begin{corollary}\label{necessity:almost_torsion-free:2} For an abelian group $G$, consider the following conditions:
\begin{itemize}
  \item[(a)] every infinite subset of $G$ is potentially dense in $G$,
  \item[(b)] every proper unconditionally closed subset of $G$ is finite (that is, $\Mar_G$ coincides with the  cofinite topology of $G$),
  \item[(c)] every proper algebraic subset of $G$ is finite   (that is, $\Zar_G$ coincides with the cofinite topology of $G$).
\end{itemize}
Then (a)~$\to$~(b)~$\leftrightarrow$~(c). Moreover, all three conditions are equivalent whenever $|G|\leq \cont$.
\end{corollary}

\begin{proof} To show (a)~$\to$~(b) assume that $X$ is an infinite unconditionally closed subset of $G$. Then $X$ must be  simultaneously closed and dense in some Hausdorff group  topology on $G$. Thus, $X=G$.

The equivalence (b)~$\leftrightarrow$~(c) follows from Corollary \ref{Perelman}.

The  last assertion is an obvious consequence of Theorem D.
\end{proof}

\begin{remark} 
Since item (a) of Corollary \ref{necessity:almost_torsion-free:2} yields $|G|\leq 2^\cont$, it is not possible to invert the implication (a)~$\to$~(b) in Corollary \ref{necessity:almost_torsion-free:2} if $|G|>2^\cont$. In our forthcoming paper \cite{DS_Kronecker}, we invert it for groups of size at most $2^\cont$. 
\end{remark}

From Theorems D and  \ref{Zariski:dense:sets:in:unbounded:groups} we obtain the following:

\begin{corollary}
Let $G$ be an unbounded abelian group such that $|G|\le\cont$. Then an infinite subset $X$ of $G$ is potentially dense in $G$ if and only if $mX$ is infinite for every $m\in\mathbb{N}\setminus\{0\}$.
\end{corollary}

The potentially dense subsets of almost torsion-free or divisible abelian groups of arbitrary size are described in \cite{DS_HMP}. 

\section{Open questions}
\label{open:questions}

Since our results provide a sufficiently clear picture in the abelian case, we include a list of questions concerning the possibilities to extend some of them in the non-abelian case. 

Theorem A leaves the following question open.

\begin{question}\label{question1}
Which of the equalities $\Zar_G = \Mar_G = {\mathfrak P}_G$  from Theorem A  remain true for nilpotent groups?
\end{question}
  
According to Bryant's  theorem, $\Zar_G$ is Noetherian for every  abelian group $G$. This fails to be true in general,  e.g.,  there exist infinite (necessarily non-abelian) groups $G$ with discrete $\Zar_G$; see Example \ref{discrete:Zariski}. Nevertheless, there is a huge gap between Noetherian and discrete topologies. In fact,  being Noetherian is a much stronger property than compactness: it is easy to see that  a space is  Noetherian if and only if all its subspaces are compact. This justifies the following question.  

\begin{question}\label{question2} Let $G$ be a group. If $\Zar_G$ is  compact,  must   $\Zar_G$   be necessarily Noetherian?
\end{question}

According to Corollary \ref{Zariski:topology:never:group}, the Zariski topology $\Zar_G$ of an infinite abelian group is never Hausdorff, while it is Noetherian by Bryant's  theorem. This motivates the following question.  

\begin{question}\label{question3}
Does there exist an infinite group $G$ such that its Zariski topology $\Zar_G$ is both
compact and Hausdorff?
\end{question}
 
 Let us note that it was necessary to relax ``Noetherian" to ``compact'' in the above question, since an infinite  Noetherian space is never Hausdorff; see
Fact \ref{Noetherian:facts}. In particular, a positive answer   to this question  would provide a negative answer to Question \ref{question2}. 
 
  Finally, let us recall two questions from \cite{DS_OPIT}. 
 
 \begin{question}\label{QuesA}
Let $G$ be a group of size at most $2^\cont$ and $\mathcal E$ be a countable family of subsets of $G$. Can one find a Hausdorff group topology ${\mathcal T}_{\mathcal E}$ on $G$ such that the ${\mathcal T}_{\mathcal E}$-closure of every $E\in{\mathcal E}$ coincides with its ${\mathfrak M}_G$-closure?
\end{question}

For an Abelian group $G$ with $|G| \leq \cont$ the answer is positive, and in fact the topology ${\mathcal T}_{\mathcal E}$ in this case can be chosen to be precompact by Theorem \ref{Main:theorem}.

The counterpart of Question \ref{QuesA} for ${\mathfrak Z}_G$    instead of ${\mathfrak M}_G$ has a consistent  negative answer; see the comment in \cite{DS_OPIT}. 

Let us consider now the counterpart of Question \ref{QuesA} for ${\mathfrak P}_G$ instead of ${\mathfrak M}_G$.

\begin{question}
\label{QuesB}
Let $G$ be a group of size at most $2^\cont$ having at least one precompact Hausdorff group topology, and let $\mathcal E$ be a countable family of subsets of $G$. Can one find a precompact Hausdorff group topology ${\mathcal T}_{\mathcal E}$ on $G$ such that the ${\mathcal T}_{\mathcal E}$-closure of every $E\in{\mathcal E}$ coincides with its ${\mathfrak P}_G$-closure?
\end{question}

Again, for an Abelian group $G$ with $|G| \leq \cont$, the answer is positive by Theorem \ref{Main:theorem}.

\bigskip
\noindent{\bf Historic remark.}
The principle results of this paper were announced by the first author in his keynote address during the 10th Prague Topological Symposium TOPOSYM 2006 (August 13--19, 2006, Prague, Czech Republic), and were also mentioned in  \cite[Section 5]{DS_OPIT}.

\bigskip
\noindent{\bf Acknowledgment.} 
We are grateful 
to Daniel Toller for his careful reading and helpful comments.

\end{document}